\documentclass{amsart}
\usepackage{xcolor}
\usepackage{enumerate}
\usepackage{amsmath}
\usepackage{mathrsfs}
\usepackage{booktabs,caption}
\usepackage{ulem}
\usepackage{amssymb}
\usepackage{mathtools}
\usepackage{color}
\usepackage{cite,url}
\usepackage[toc,page]{appendix}

\newcommand{\Rni}{\mathbb{R}^{n_{i}}}

\newcommand{\nfR}{\mathbb{R}}
\newcommand{\xpi}{x_i}

\newcommand{\xmi}{x_{-i}}

\newcommand{\fpi}{f_{i}}

\newcommand{\lip}{\left<}
\newcommand{\rip}{\right>}
\newcommand{\lvt}{\left[}
\newcommand{\rvt}{\right]}

\newcommand{\sing}[1]{{#1}_{sing}}
\newcommand{\tdx}{\tilde{x}}
\newcommand{\tdu}{\tilde{u}}
\newcommand{\ddd}{,\ldots,}
\newcommand{\mf}{\mathfrak}

\newcommand{\fall}{\quad \mbox{for all}\quad}

\newcommand{\wt}{\widetilde}
\newcommand{\td}{\tilde}
\newcommand{\wtU}{\widetilde{\mathcal{U}}}
\newcommand{\wtV}{\widetilde{\mathcal{V}}}
\newcommand{\wtW}{\widetilde{\mathcal{W}}}
\newcommand{\mbF}{\mbox{F}}
\newcommand{\re}{\mathbb{R}}
\newcommand{\cpx}{\mathbb{C}}

\newcommand{\N}{\mathbb{N}}

\newcommand{\Pj}{\mathbb{P}}

\def\gm{\gamma}

\newcommand{\st}{\mathit{s.t.}}
\newcommand{\reff}[1]{(\ref{#1})}

\newcommand{\lmd}{\lambda}
\newcommand{\pt}{\partial}
\newcommand{\dt}{\delta}

\newcommand{\mc}[1]{\mathcal{#1}}

\newcommand{\deq}{\,\coloneqq\,}
\def\rank{\mbox{rank}}
\def\cod{\mbox{codim}\,}
\newcommand{\bdes}{\begin{description}}
\newcommand{\edes}{\end{description}}

\newcommand{\bal}{\begin{align}}
\newcommand{\eal}{\end{align}}

\newcommand{\bnum}{\begin{enumerate}}
\newcommand{\enum}{\end{enumerate}}

\newcommand{\bit}{\begin{itemize}}
\newcommand{\eit}{\end{itemize}}

\newcommand{\bea}{\begin{eqnarray}}
\newcommand{\eea}{\end{eqnarray}}
\newcommand{\be}{\begin{equation}}
\newcommand{\ee}{\end{equation}}

\newcommand{\baray}{\begin{array}}
\newcommand{\earay}{\end{array}}

\newcommand{\bsry}{\begin{subarray}}
\newcommand{\esry}{\end{subarray}}

\newcommand{\bca}{\begin{cases}}
\newcommand{\eca}{\end{cases}}

\newcommand{\bcen}{\begin{center}}
\newcommand{\ecen}{\end{center}}

\newcommand{\bbm}{\begin{bmatrix}}
\newcommand{\ebm}{\end{bmatrix}}

\newcommand{\btab}{\begin{tabular}}
\newcommand{\etab}{\end{tabular}}

\theoremstyle{definition}
\newtheorem{theorem}{Theorem}[section]

\newtheorem{proposition}[theorem]{Proposition}

\newtheorem{lemma}[theorem]{Lemma}

\newtheorem{example}[theorem]{Example}

\newtheorem{remark}[theorem]{Remark}

\numberwithin{equation}{section}

\begin{document}

\title[Algebraic Degrees of GNEPs]
{Algebraic degrees of generalized Nash equilibrium problems}

\author{Jiawang Nie}
\author{Kristian Ranestad}
\author{Xindong Tang}

\address{Jiawang Nie, Department of Mathematics,\\
University of California San Diego,
9500 Gilman Drive, La Jolla, California, {\rm92093}, USA}
\email{njw@math.ucsd.edu}

\address{Kristian Ranestad,
Department of Mathematics,\\
University of Oslo,
PB 1053 Blindern, {\rm0316} Oslo, Norway}
\email{ranestad@math.uio.no}

\address{Xindong Tang, Department of Mathematics,\\
Hong Kong Baptist University,
Kowloon Tong, Kowloon, Hong Kong.}
\email{xdtang@hkbu.edu.hk}

\begin{abstract}
This paper studies the algebraic degree of
generalized Nash equilibrium problems (GNEPs) given by polynomials.
Their generalized Nash equilibria (GNEs),
as well as their KKT or Fritz-John points,
are algebraic functions in the coefficients of defining polynomials.
We study the degrees of these algebraic functions,
which also count the numbers of complex KKT or Fritz-John points.
Under some genericity assumptions, we show that
a GNEP has only finitely many complex Fritz-John points
and every Fritz-John point is a KKT point.
We also give formulae for algebraic degrees of GNEPs,
which count the numbers of complex Fritz-John points for generic cases.
\end{abstract}



\maketitle

\section{Introduction}
We consider generalized Nash equilibrium problems (GNEPs).
Suppose there are $N$ players and the $i$th player's strategy vector is
$x_{i}\in\mathbb{R}^{n_{i}}$ (the $n_i$-dimensional real Euclidean space).
Denote
\[x_i \coloneqq  (x_{i,1},\ldots,x_{i,n_i}),\quad x \coloneqq (x_1,\ldots,x_N).\]
The total dimension of all strategy vectors is $n  \coloneqq  n_1+ \ldots + n_N.$
When the $i$th player's strategy is considered,
we use $x_{-i}$ to denote the subvector of all players' strategies
except the $i$th one, i.e.,
\[
x_{-i} \,  \coloneqq  \, (x_1, \ldots, x_{i-1}, x_{i+1}, \ldots, x_N).
\]
For convenience, we also write $x=(x_{i},x_{-i})$ when $x_i$ is considered.
The main purpose of the GNEP is to find a tuple of strategies
$u = (u_1, \ldots, u_N)$
such that each $u_i$ is a minimizer of the $i$th player's optimization
\be
\label{eq:GNEP}
\mbox{F}_i(u_{-i}): \quad
\left\{ \begin{array}{cl}
\min\limits_{\xpi\in \Rni}  &  \fpi(x_i,\xmi) \\
\st & g_{i,j}(x_i,\xmi)  = 0 \ (j\in\mc{E}_{i,1}), \\
    & g_{i,j}(x_i,\xmi)  \ge 0 \ (j\in\mc{E}_{i,2}).
\end{array} \right.
\ee
In the above, the $\mc{E}_{i,1}$ and $\mc{E}_{i,2}$
are disjoint labeling sets (possibly empty),
the $f_i$ and $g_{i,j}$ are continuous functions in $x$. Note that
$\xmi \,  \coloneqq  \, (x_1, \ldots, x_{i-1}, x_{i+1}, \ldots, x_N).$
The $i$th player's optimization is parameterized by other players' strategies.
A strategy tuple $x = (x_1, \ldots, x_N)$ satisfying the above
is called a generalized Nash equilibrium (GNE).

When all the defining functions $f_i$ and $g_{i,j}$ are polynomials in $x$ with real coefficients,
the GNEP is called a generalized Nash equilibrium problem of polynomials.
If for every player $i$, each constraining function $g_{i,j}$
depends only on $x_i$,
i.e., the $i$th player's feasible strategy set is
independent of other players' strategies,
then the GNEP is call a {\it Nash equilibrium problem} (NEP)
and the corresponding $u$ is called a {\it Nash equilibrium} (NE).

GNEPs are generalizations of Nash equilibrium problems \cite{Nash1951}.
They have been widely used in broad areas,
such as marketing, supply chain
management, telecommunications, and machine learning.
We refer to \cite{Liu2016,
Nagurney2021,Pang2008}
for recent applications of GNEPs.
The existence of GNEs under some continuity and convexity assumptions
is given in \cite{Aussel2021}.
For GNEPs of twice differentiable continuous functions,
generic structural properties are studied in \cite{Dorsch2013,Dorsch2013structure}.
It is typically quite a difficult question to solve GNEPs.
Some computational methods are given in \cite{GurPang09,Facchinei2009}.
In \cite{Nie2020gs,Nie2020nash,Nie2021convex,Nie2021rational},
Moment-SOS relaxation methods are given for solving NEPs and GNEPs.
We refer to \cite{Facchinei2010,Facchinei2010book} for general surveys on GNEPs.

\subsection{KKT and Fritz-John conditions}
\label{sc:optconds}

Consider the $i$th player's optimization $\mbox{F}_i(x_{-i})$,
which is parameterized by the strategies of the other players 
\[x_{-i} = (x_1, \ldots, x_{i-1}, x_{i+1}, \ldots, x_N).\]
At a feasible point $x_i$, its active labeling set $E_i$ is
\[
E_i \, \coloneqq \, \{j \in \mc{E}_{i,1}  \cup \mc{E}_{i,2} : \,
g_{i,j}(\xpi,\xmi)=0 \}.
\]
Clearly, $\mc{E}_{i,1} \subseteq E_i \subseteq\mc{E}_{i,1} \cup \mc{E}_{i,2}$.
When the linear independence constraint qualification (LICQ) holds (i.e.,
the gradient set $\{\nabla_{\xpi}g_{i,j}(\xpi,\xmi):j \in E_i\}$
is linearly independent), if $x_i$ is a minimizer of the optimization
$\mbox{F}_i(x_{-i})$,
then there exist Lagrange multipliers
$\lmd_{i,j}$ ($j \in E_i$)
such that the {\it Karush-Kuhn-Tucker} (KKT) conditions hold:
\be \label{eq:KKTwithLM}
\left\{
\begin{array}{c}
  \nabla_{x_i} f_i(x)-\sum\limits_{ j \in E_i  }
       \lambda_{i,j}\nabla_{x_i} g_{i,j}(x)=0,  \\
  g_{i,j}(x)=0\,(j  \in E_i).  \\
\end{array}
\right.
\ee
In the above, we do not mention the sign conditions
$\lambda_{i,j}\ge0\, (j\in E_i\cap\mc{E}_{i,2})$,
because they are not relevant to algebraic degrees of GNEPs.
When the LICQ fails, the KKT conditions may or may not hold.
However, no matter if the LICQ holds or not,
there always exists a new variable $\lambda_{i,0}$ such that
\be
\label{eq:FJwithLM}
\left\{
\begin{array}{c}
  \lambda_{i,0}\nabla_{x_i} f_i(x)-\sum\limits_{j \in E_i}
           \lambda_{i,j}\nabla_{x_i} g_{i,j}(x)=0,\\
    g_{i,j}(x)=0\,(j  \in E_i) , \\
 \mbox{not all} \,\, \lmd_{i,0} \,\,
 \mbox{and} \,\,  \lambda_{i,j} \,\, \mbox{are} \,\, 0.
\end{array}
\right.
\ee
A point $x_i$ satisfying \reff{eq:KKTwithLM} is called
a KKT point and $x_i$ satisfying \reff{eq:FJwithLM} is called
a Fritz-John (FJ) point, for the optimization $\mbox{F}_i(x_{-i})$.
The system \reff{eq:KKTwithLM} is called the KKT system and
\reff{eq:FJwithLM} is called the Fritz-John system.
It is important to note that every minimizer of
$\mbox{F}_i(x_{-i})$ is a Fritz-John point. When $\lambda_{i,0}\ne 0$,
the Fritz-John conditions imply the KKT conditions.
We refer to \cite{Bert97} for KKT and Fritz-John conditions in optimization.
A tuple $x=(x_1\ddd x_N)$ is said to be a {\it Fritz-John point}
for the GNEP if $x$ satisfies (\ref{eq:FJwithLM}) for each $i=1\ddd N$.
Furthermore, we remark that if we re-enumerate the active label set as
$E_i \coloneqq \{1, \ldots, m_i \}$,
then $x\in\cpx^n$ satisfies \reff{eq:FJwithLM} if and only if
\be \label{eq:jacdef}
\left\{ \baray{c}
\rank \, \bbm \nabla_{\xpi}f_i(x) & \nabla_{\xpi}g_{i,1}(x) & \cdots &
 \nabla_{\xpi}g_{i,m_i}(x) \ebm \le m_i, \\
g_{i,1}(x) = \cdots = g_{i,m_i}(x) = 0.
\earay \right.
\ee

The investigation of algebraic properties of GNEPs is interesting and quite important to the study of numerical methods for solving GNEPs.
For instance, the Lagrange multiplier expression method introduced in \cite{Nie2021convex} is guaranteed to solve the convex GNEP by solving some semidefinite programs, under the assumption that the LICQ is satisfied at all its FJ points (see \cite[Proposition~3.6, Theorem~5.2]{Nie2021convex}).
Given the FJ point $x$ with its active label set denoted as $E_i \coloneqq \{1, \ldots, m_i \}$,
the LICQ is equivalent to the nonsingularity of the Jacobian matrix
\[\bbm \nabla_{\xpi}g_{i,1}(x) & \cdots &
 \nabla_{\xpi}g_{i,m_i}(x) \ebm.\]
As shown in \cite[Example~3.8]{Nie2021convex}, there exist GNEPs where the LICQ fails to hold at some GNEs (hence FJ points).
On the other hand, an FJ point is a KKT point if the LICQ holds at it.
For a convex GNEP, every KKT point is a GNE, yet an FJ point may or may not be a solution.
Therefore, an interesting question is, when the GNEP is given by generic polynomials, is it true or not that all its Fritz-John points satisfy the LICQ?

Similar to nonconvex optimization problems, when there is no convexity assumption, not all KKT points are GNEs.
The semidefinite relaxation methods in \cite{Nie2020nash,Nie2021rational} apply feasible extensions to preclude KKT points that are not GNEs from the candidate solution set in each iteration.
Theoretically, the finiteness of the KKT point set is required to guarantee a finite convergence for these semidefinite relaxation methods.
Furthermore, the cardinality of the KKT point set, if it is finite, gives an upper bound for the number of iterations that the methods in \cite{Nie2020nash,Nie2021rational} guarantee to find a GNE or detect nonexistence of solutions.
Moreover, the finiteness and the cardinality of KKT point sets are also closely related to the effectiveness of the polyhedron homotopy method for solving GNEP in \cite{lee2023polyhedral}.
To this end, one wonders when the GNEP is given by generic polynomials, does it only have finitely many KKT points?
In addition, can we get a closed formulate for the number of KKT points in dimensions of strategy vectors and degrees of defining polynomials?
\subsection{Contributions}

This paper studies sets of KKT and Fritz-John points for GNEPs of polynomials.
We mainly focus on the finiteness for the set of KKT and Fritz-John points
under genericity assumptions,
and the number of KKT and Fritz-John points when there are finitely many of them.
The major results of this paper are:

\begin{itemize}

\item We study the Fritz-John system for the GNEP given by polynomial functions.
Under some genericity assumptions, we show that
the GNEP only has finitely many complex Fritz-John points.
Moreover, when defining polynomials are generic,
we show that every Fritz-John point is a KKT point,
and thus every GNE is a KKT point.

\item We give a formula for the algebraic degree of GNEPs
when defining polynomials are generic,
which counts the number of complex Fritz-John points,
and provides an upper bound for the iteration loops of
the computational methods in \cite{Nie2020nash,Nie2021rational}.

\item For non-generic GNEPs of polynomials,
we give an upper-bound for the number of complex Fritz-John points,
when there are finitely many of them.
Some computational examples are also given
to confirm the degree formulae.
\end{itemize}

This paper is organized as follows.
In Section~\ref{sc:pre},
we review the optimality conditions for GNEPs,
and introduce some basics of complex Fritz-John points.
In Section~\ref{sc:finite},
we study the finiteness of complex Fritz-John points for GNEPs,
under some genericity assumptions.
We consider multi-projective varieties
and their algebraic degrees in Section~\ref{sc:MPV}.
The algebraic degree formulae are given in Section~\ref{sc:degree}.
Some symbolic computational results are presented in Section~\ref{sc:ce}.

\section{Preliminaries}
\label{sc:pre}

\subsection*{Notation}
The symbol $\mathbb N$ (resp., $\mathbb Q$, $\mathbb R$, $\mathbb C$) stands for the set of
nonnegative integers (resp., rational numbers, real numbers, complex numbers).
For the integer $k\in \N$, denote the set $[k]  \coloneqq  \{1, \ldots, k\}$ if $k\ne0$,
and $[k]:=\emptyset$ if $k=0$.
We use $e_i$ to denote the vector such that the $i$th entry is
$1$, and all others are zeros.
For a nonnegative integer vector $a = (a_1, \ldots, a_k) \in \N^k$,
denote that  $|a| \coloneqq a_1 + \cdots + a_k$.
Let $a\deq(a_1\ddd a_k)$ and $b\deq(b_1\ddd b_k)$
be tuples of nonnegative integers,
and let $a\le b$ mean that $a_i\le b_i$ for each $i=1\ddd k$.
Let $\nfR[x]$ denote the ring of polynomials
with real coefficients in $x$,
and $\nfR[x]_d$ denote its subset of polynomials
whose degrees are not greater than $d$.
The notation $\mathbb{Q}[x]$, $\cpx[x]$ and $\cpx[x]_d$ are similarly defined.
For the $i$th player's strategy vector $x_i$,
the notation $\re[x_i]$ and $\re[x_i]_d$
are defined in the same way.
For the $i$th player's objective $f_i(x)$,
the $\nabla_{x_i}f_i$ means its gradient with respect to $x_i$.

In the following, we review some basic results in algebraic geometry.
Let $\cpx$ be the complex field and $\cpx[x]$ be the ring of polynomials
in the variable $x = ( x_{i,j} ).$
An {\it ideal} $I$ of $\cpx[x]$ is a subset of $\cpx[x]$ such that
$a+b \in I$ for all $a, b \in I$ and
$q\cdot p\in I$ for all $p\in I$ and $q\in \cpx[x]$.
A polynomial tuple $(p_1 \ddd p_m)$ generates the ideal
\[
\lip p_1 \ddd p_m \rip \, \coloneqq \,
\{ p_1q_1+\dots+p_mq_m: q_1\ddd q_m\in\cpx[x] \} .
\]
Every ideal is generated by a finite set of polynomials.
This is Hilbert's basis theorem (see \cite{cox2013book}).
For an ideal $I$, the set
\[
V(I) \, \coloneqq \,
\{x\in\cpx^n: p(x) = 0 \, \forall \, p \in I
\}
\]
is called the {\it affine variety of} $I$.
For a set $V \subseteq \cpx^n$, the polynomial set
\[
I(V) \,\coloneqq \,
\{p\in\cpx[x]: p(x)=0 \mbox{ for all } x\in V \}
\]
is an ideal of $\cpx[x]$.
It is called the {\it vanishing ideal of  $V$}.
There exists extensive work about optimization with
polynomials and varieties (see
\cite{MomentSOShierarchy,Las01,LasBk15,LasICM,Lau09,LauICM,Scheid09}).
Sum of squares polynomials and matrices are useful techniques for
solving polynomial optimization (see \cite{HilNie08,PMI11,Nie2023moment,SOSsph12,nieopcd}).
They are also useful for solving tensor optimization
(see \cite{NYZ18,NieZhang18,Nuclear17}).
Moreover, ideals and varieties are recently exploited in the research of games
(see \cite{Datta2003,Portakal2022,Sturmfels2002}).

For the positive integer $n$, the $n$-dimensional projective space $\Pj^n$
is the set of all lines passing through the origin of $\cpx^{n+1}.$
A point $\td{z}$ in $\Pj^n$ has the coordinate
$[z_0 : z_1 : \cdots :z_n]$ such that at least one of $z_j$ is nonzero.
The coordinates of $\td{z}$ are unique up to a nonzero scaling of
$\td{z} = (z_0, z_1, \ldots, z_n)$.
Let $(p_1 \ddd p_m)\subseteq \cpx[\td{z}]$
be a tuple of homogeneous polynomials in $\td{z}$,
the ideal $I \coloneqq \lip p_1 \ddd p_m\rip$ is called a {\it homogeneous ideal}.
It determines the projective algebraic variety in $\Pj^n$:
\[
U \, \coloneqq \,
\{\td{z}\in\Pj^n: p_1(\td{z})=\dots=p_m(\td{z})=0\}.
\]
For positive dimensions $n_1 \ddd n_N$ such that $n_1+\dots+n_N=n$,
the Cartesian product $\Pj^{n_1}\times\dots\times\Pj^{n_N}$
is called a multi-projective space.
For the tuple of positive integers $\nu\,\coloneqq\,(n_1\ddd n_N)$, denote
\[\mathbb{P}^\nu\,\coloneqq\,\mathbb{P}^{n_1}\times\dots\times\mathbb{P}^{n_N}.\]
Given the tuple $\td{x} \coloneqq (\td{x}_1 \ddd \td{x}_N)$ of vector variables,
a polynomial $p \in\cpx[\td{x}]$ is {\it multi-homogeneous} in
$(\td{x}_1 \ddd \td{x}_N)$
if for every $i\in [N]$,
the polynomial $p(\td{x}_i,\td{x}_{-i})$ is homogeneous in $\td{x}_i$.
Similarly, an ideal is said to be multi-homogeneous
if it is generated by a set of multi-homogeneous polynomials. A set like
\[
W = \{\td{x}\in\Pj^{\nu}:
p_1(\td{x})=\dots=p_m(\td{x})=0 \}
\]
is called an {\it multi-projective} variety
if every $p_j$ is multi-homogeneous.

For the affine (resp., projective, multi-projective) space,
a {\it hypersurface} is given by a single
(resp., homogeneous, multi-homogeneous) polynomial equation.
A {\it hyperplane} is given by a single linear equation.
In the {\it Zariski topology}, closed sets are varieties
and open sets are complements of varieties.
A variety $V$ is {\it irreducible} if it is not a union
of two distinct proper subvarieties.
Every variety is a union of finitely many irreducible subvarieties.
Each of these irreducible subvarieties is called an {\it irreducible component}.
Throughout the paper, a property is said to hold {\it generically}
if it holds for all points in the space of input data
except a set of Lebesgue measure zero.

The {\it dimension} of an irreducible projective variety $V$
is defined to be the length $k$ of the longest chain of irreducible subvarieties
$V=V_0\supsetneq V_1 \supsetneq\dots\supsetneq V_k$.
For an irreducible multi-projective variety
$W\subseteq\Pj^{n_1}\times\dots\times\Pj^{n_N}$,
the dimension of $W$ equals the largest number $k$
such that there exist nonnegative integers $l_1 \ddd l_N$, with each
$l_i\le n_i$ and $l_1+\dots +l_N=k$, satisfying that
(recall that $[l_i]=\emptyset$ if $l_i=0$)
\be \label{eq:Wintersect}
\left\{ (\td{x}_1, \ldots, \td{x}_N) \in W \left| \,
\begin{array}{c}
(u_{i,j})^T \td{x}_i =0, \\
\forall i \in [N] ,\,  j \in [l_i]
\end{array}\right.\right\}
\ee
is nonempty for all $u_{i,j} \in \Pj^{n_i}$.
The dimension of a variety equals the biggest dimension of its irreducible subvarieties,
and the {\it codimension} of a $k$-dimensional variety in $\Pj^{\nu}$ equals $n-k$.
For the variety $W$, its dimension $\dim  W=0$ if and only if
its cardinality $|W|<\infty$.
A variety is said to have a {\it pure} dimension if
all of its irreducible components have the same dimension.
A pure $k$-dimensional variety $W$
is said to be a {\it complete intersection}
if its vanishing ideal is generated by $n-k$ polynomials.
If $W$ is an irreducible variety of dimension $k$ and a homogeneous $p$
is not identically zero on $W$, then the zero locus of $p$ in $W$
has dimension $k-1$ (see Section~1.6.1 of \cite{Shafarevich2013}).
Moreover, two multi-projective varieties $W_1$, $W_2$ in $\Pj^{\nu}$
are said to {\it intersect properly} if
\[
\dim {W_1\cap W_2} \,= \, \dim W_1+\dim W_2-n.
\]
For a pure $k$-dimensional projective variety $V$,
its algebraic degree is the number of points in the intersection
of $V$ and $k$ generic hyperplanes.
The algebraic degree of a multi-projective variety
$W  \subseteq  \Pj^{n_1}\times\dots\times\Pj^{n_N}$
is an array $\deg W$ indexed by tuples $l=(l_1 \ddd l_N) \in \N^N$ with $l_1+\dots+l_N = \dim \, W$ and $l_i\le n_i$ for all $i\in[N]$,
such that $(\deg W)_l$ counts the number of points in the intersection given by (\ref{eq:Wintersect}) for generic $u_{i,j}\in\Pj^{n_i}$ \cite{VDW1978}.
In particular, when $\dim W = 0$, the algebraic degree $\deg W$
counts the number of points in $W$.

Smoothness is useful when we consider algebraic degrees.
For an irreducible multi-projective variety $W$,
whose vanishing ideal is defined by multi-homogeneous polynomials $p_1\ddd p_m$,
its {\it(multi-projective) tangent space} at a point $\td{u}\in W$ is
\[\mathbb{T}_{\td{u}}W \deq \left\{\td{x}\in\Pj^{\nu} \,:\,
\bbm \nabla_{\td{x}} p_1(\td{u}) & \cdots & \nabla_{\td{x}} p_m(\td{u}) \ebm^T
\cdot \td{x} = 0 \right\}.\]
The dimension of $\mathbb{T}_{\td{u}}W$ is at least $\dim  W$ for every $u\in W$
(see \cite[Theorem~2.3]{Shafarevich2013}).
The variety $W$ is said to be smooth at $\td{u}$ if
\be\label{eq:smooth}
\dim  \mathbb{T}_{\td{u}}W = \dim W.
\ee
In other words,
the $\td{u}$ is a smooth point if and only if the rank of
$J(\td{u})$ equals the codimension of $W$.
When $W$ is reducible, the smoothness for $\td{u}$ is defined locally.
That is, if we let $W_u$ be the irreducible component
which has the largest dimension over all irreducible components containing $\td{u}$,
then we say $\td{u}$ is a smooth point of $W$ if it is a smooth point of $W_u$.
A point in $W$ is called {\it singular} or {\it nonsmooth} if it is not a smooth point of $W$,
and the set of all singular points of $W$ is denoted as $\sing{W}$.
Moreover, the $W$ is said to be smooth if $\sing{W}=\emptyset$.
For two projective varieties $W_1$ and $W_2$,
we say they {\it intersect transversely at
$\td{u}\in (W_1\cap W_2)\setminus (\sing{(W_1)}\cup\sing{(W_2)})$} if
\[
\dim \, (\mathbb{T}_{\td{u}}W_1+\mathbb{T}_{\td{u}}W_2)
\,  = \, \dim  \Pj^{n_1}\times\dots\times\Pj^{n_N}\, =\, n.
\]
For every irreducible component $W'$ in $W_1\cap W_2$,
if the intersection
\[
W'\cap \big( \sing{(W_1)}\cup\sing{(W_2)} \big)
\]
is a proper subset of $W'$,
and $W_1$ and $W_2$ intersect transversely at every point in
\[
(W_1\cap W_2)\setminus \big( \sing{(W_1)}\cup\sing{(W_2)} \big),
\]
then they are said to {\it intersect transversely}.
If $W_1\cap W_2 =\emptyset$, then they intersect transversely.
In general, two varieties $W_1$ and $W_2$ intersect transversely if
$\cod W_1+\cod W_2=\cod (W_1\cap W_2)$
and their intersection is smooth outside their singular loci.
We refer to \cite{Shafarevich2013,Harris1992,Ful}
for more details on smoothness and transversal intersections for varieties.

Bertini's Theorem considers the smoothness of intersections.
The following result from \cite{Nie2009algebraic}, as a corollary of Bertini's Theorem,
is frequently used in this paper, and we refer to \cite{Harris1992} for more details about smoothness.
\begin{proposition}[See {\cite[Theorem A.1]{Nie2009algebraic}}]\label{pr:bertini}
Let $W$ be a $k$-dimensional multi-projective variety in $\Pj^{\nu}$, and let $\Pj^m$ be a projective space parameterizing hypersurfaces in $\Pj^{\nu}$.
Let $Z = \cap_{\mc{H} \in \Pj^m} \mc{H}$ be the common points of these hypersurfaces.
Then for a generic $\mc{H}\in \Pj^m$, the intersection $Y: = W \cap \mc{H}$ is a multi-projective variety of dimension $k-1$,
and \[\sing{(Y)}\subseteq(\sing{W}\cap Y)\cup Z.\]
\end{proposition}

By Proposition~\ref{pr:bertini},
if we let $p_1(\tdx)\ddd p_k(\tdx)$ be generic\footnote{In our context, we say that a multi-homogeneous polynomial is generic if it is generic in the space of multi-homogeneous polynomials in $\tdx$ for a fixed multi-degree.}
multi-homogeneous polynomials in $\Pj^{\nu}$,
then hypersurfaces defined by each $p_i(\tdx)=0$ intersect transversely,
i.e., the intersection is smooth, and its dimension equals $n-k$.
Moreover, let $W\subseteq\Pj^{\nu}$ be a multi-projective variety,
and let $p_1,\dots,p_k$ be multi-homogeneous polynomials such that the coefficients of each $p_i$ are parameterized by $\beta_i\in\Pj^{m_i}$ .
By implementing Proposition~\ref{pr:bertini} repeatedly we may conclude:  If for every $i=1\ddd k$,
the $p_i(\tdx)=0$ does not have any fixed point in $W\cap\{\tdx\in\Pj^{\nu}: p_1(\tdx)=\dots=p_{i-1}(\tdx)=0\}$ when we vary $\beta_i$,
then $W$ intersects $\{\tdx\in\Pj^{\nu}: p_1(\tdx)=\dots=p_{k}(\tdx)=0\}$ transversely for a generic choice of $\beta_1\ddd\beta_k$.

For an affine variety $X \subseteq \cpx^{n_1} \times \cdots \times \cpx^{n_N}$,
the embedding of $X$ in the multi-projective space $\Pj^{n_1}\times\dots\times\Pj^{n_N}$ consists of all points
\[
\Pj(x) \, \coloneqq \,
(\td{x}_1,\td{x}_2\ddd \td{x}_N)
\]
with $\td{x}_i \coloneqq (1,x_{i,1},x_{i,2}\ddd, x_{i,n_i})$ and $(x_1\ddd x_N) \in X$.
The set of all $\Pj(x)$ with $x \in X$ is denoted as $\Pj(X)$.
For convenience, we say that a multi-projective variety $Y$
contains the affine variety $X$ if it contains $\Pj(X)$.

\subsection{Algebraic degrees of multi-projective varieties}
For a multi-projective variety $\mc{X}$ in
$\mathbb{P}^{\nu} \,\coloneqq\, \Pj^{n_1}\times\dots\times\Pj^{n_N}$,
the algebraic degree of $\mc{X}$ is a nonnegative integer vector,
labeled by a tuple
\[
l \,\coloneqq\, (l_1,\dots,l_N)\in\N^N, \quad
\text{with} \quad  l\le\nu, \,\,  |l|=\dim \mc{X},
\]
such that each $(\deg\mc{X})_l$ counts the number of points in the intersection
\be  \label{eq:degX}
\bigcap_{i=1}^N \Big \{\tilde{x} \in \mc{X}:
(u_{i,1})^T\tilde{x}_i=\dots
= (u_{i,l_i})^T\tilde{x}_i=0 \Big \},
\ee
where $u_{i,1},\dots,u_{i,l_i}$ are generic vectors in $\Pj^{n_i}$ \cite{VDW1978}.
For the case that $\dim \mc{X}=0$,
the $\deg \mc{X}$ is just a nonnegative integer
which counts the cardinality $|\mc{X}|$ of $\mc{X}$.

Let $k$ be a positive integer and $\delta \,\coloneqq\,
(\delta_1,\dots,\delta_N)\in\N^N$ such that $\delta_1+\dots+\delta_N=k$,
and let $\mathbb{Z}_2\deq\{0,1\}$ be the binary set.
Denote the set
\[
\mathbb{Z}_2^{[\delta\times k]} \coloneqq
\left\{
A = (A_{i,j})  \in  \mathbb{Z}_2^{N\times k} \left|
\begin{array}{l}
\sum\limits_{j = 1,\ldots, k} A_{i,j}  =\delta_i\,
(i\in[N]),\\
\sum\limits_{i = 1,\ldots, N} A_{i,j} = 1 \,
\, (j\in[k])
\end{array}
\right.\right\}.
\]
For the matrix $A \in \mathbb{Z}_2^{[\delta\times k]}$,
we denote the $j$th column by $A_{:,j}$.
Let $z\deq(z_1,\dots,z_k)$ be a tuple of vector variables with each
$z_i = (z_{i,1}, \ldots, z_{i,N}) $.
Define the polynomial in $z$
\be\label{eq:mcA}
\mc{A}_{\delta}(z_1,\dots,z_k) \,\coloneqq\,
\sum_{A\in \mathbb{Z}_2^{[\delta\times k]}  }
(z_1)^{A_{{:,1}}} (z_2)^{A_{{:,2}}} \dots (z_k)^{A_{{:,k}}}.
\ee
In the above, the $z_i^{A_{{:,i}}}\,\coloneqq\,
(z_{i,1})^{A_{1,i}}(z_{i,2})^{A_{2,i}}\dots (z_{i,N})^{A_{N,i}}$.
One may directly verify that
$\mc{A}_{\delta}(z_1,\dots,z_k)$
is the coefficient of
$t_1^{\delta_1}t_2^{\delta_2}\dots t_N^{\delta_N}$ in
\[
\prod_{i=1}^k(z_{i,1}t_1+\dots+z_{i,N}t_N),
\]
by expanding the product above.

We study the algebraic degree for the intersection of multi-projective varieties.
First, we consider complete intersections in $\Pj^{\nu}$.
For a given nonnegative integer $s$, we denote
\be\label{eq:[s]}
[s]^{(\nu)}\deq\{l\in\N^N: |l|=s,\ l\le \nu\}.
\ee
It stands for the set of all possible labels of the algebraic degree for any
$s$-dimensional multi-projective variety in $\Pj^{\nu}$.

\begin{lemma}
\label{lm:completedeg}
Let $\nu \,\coloneqq\, (n_1\,\dots,n_N)$ and $n=n_1+\ldots +n_N$.
Suppose $h_1 \ddd h_k$ are generic multi-homogeneous polynomials in
$\tilde{x} = (\td{x}_1, \ldots, \td{x}_N)$,
and the multi-degree for each $h_i$ is $d_i=(d_{i,1},\dots,d_{i,N})$.
Consider the multi-projective variety
\[\mc{H} \,\coloneqq\, \{\tilde{x}\in\mathbb{P}^{\nu}:
h_1(\tilde{x})=\dots=h_k(\tilde{x})=0 \}.
\]

{\rm(i)} If $n=k$, then $\dim \mc{H}=0$ and
\[ \deg{\mc{H}}=\mc{A}_{\nu}(d_1,\dots,d_k).\]

{\rm(ii)}If $n>k$, then $\dim \mc{H}=n-k$ and
for each $l\in [n-k]^{(\nu)}$,
\[
  (\deg{\mc{H}})_{l}=\mc{A}_{\nu-l}(d_1,\dots,d_k).
\]
\end{lemma}
\begin{proof}
(i) When $k=n$, the $\mc{H}$ is a zero-dimensional multi-projective variety, since all $h_1,\dots,h_k$ are generic.
The algebraic degree of $\mc{H}$ equals the coefficient of $t^{\nu}$ in the product
\be \label{eq:linformdeg}
\prod_{i=1}^n(t_1d_{i,1}+\dots+t_Nd_{i,N}).
\ee
This is shown in \cite[Chapter~4, Section~2.1]{Shafarevich2013}.
The coefficient of $t^{\nu}$ in (\ref{eq:linformdeg}) coincides with $\mc{A}_{\nu}(d_1,\dots,d_k)$.

(ii) When $k< n$ and $h_1,\dots,h_k$ are generic,
the variety $\mc{H}$ is smooth and $\dim \mc{H}=n-k$,
by Proposition~\ref{pr:bertini}.
For $l=(l_1,\dots,l_N)\in\N^N$ such that $|l|=n-k$,
the $(\deg {H})_l$ counts the number of points in the intersection
(recall that $[l_i] = \emptyset$ if $l_i=0$)
\be\label{eq:degH}
\bigcap_{i=1}^N  \Big\{\tilde{x} \in \mc{H}:
(u_{i,j})^T\tilde{x}_N =0,\,\forall j\in[l_i] \Big\},
\ee
for generic vectors $u_{i,j}\in\Pj^{n_i}$.
The multi-degree of the linear form $u_{i,j}^T\tilde{x}_i$
is the unit vector $e_i$.
Furthermore, since these linear forms are generic and $\dim \mc{H}=l_1+\dots+l_N$,
the intersection $\mc{H}_l$ is zero dimensional.
By (i), its algebraic degree is the coefficient of $t^{\nu}$ in the product
\[t_1^{l_1}\cdots t_N^{l_N}\cdot \prod_{i=1}^k(t_1d_{i,1}+\dots+t_Nd_{i,N}),\]
which equals the coefficient of $t^{\nu-l}$ in
$\prod_{i=1}^k(t_1d_{i,1}+\dots+t_Nd_{i,N}).$
So, we get $(\deg{\mc{H}})_{l}=\mc{A}_{\nu-l}(d_1,\dots,d_k)$.
\end{proof}

Let $\mc{X}$ and $\mc{Y}$ be two multi-projective varieties in $\mathbb{P}^{\nu}$
such that $\dim \mc{X}+\dim \mc{Y}=n$ and their intersection is transversal.
Then the algebraic degree for the intersection $\mc{X}\cap\mc{Y}$ is given by
(the set $[\dim\mc{X}]^{(\nu)}$ is given as in (\ref{eq:[s]}))
\be\label{eq:multibezout}
\deg \mc{X}\cap\mc{Y} \,= \,
\sum_{l\in[\dim\mc{X}]^{(\nu)}}(\deg (\mc{X}))_{l}
\cdot(\deg (\mc{Y}))_{\nu-l}.
\ee
This is a generalization of the Bez\'{o}ut's Theorem to the multi-projective spaces (see \cite{VDW1978} for more details).
Furthermore, we have the following useful lemma:

\begin{lemma}  \label{lm:interdeg}
Suppose the $\mc{X}$ and $\mc{Y}$ are two multi-projective varieties in
$\mathbb{P}^{\nu}$ such that $\cod \mc{X}=k_1$ and $\cod \mc{Y}=k_2$,
and both of them are equidimensional\footnote{We say a variety is {\it equidimensional} if all of its irreducible components have the same dimension.}.
If the intersection $\mc{X}\cap\mc{Y}$ is transversal,
then $\dim  \mc{X}\cap\mc{Y}=n-k_1-k_2$, and for each $l\in [n-k_1-k_2]^{(\nu)}$,
\[
(\deg \mc{X}\cap\mc{Y})_l  \quad =
\sum_{\substack{l^{(1)}\in[n-k_1]^{(\nu)}, \\ l^{(1)}\ge l}}
(\deg \mc{X})_{l^{(1)}}
\cdot(\deg \mc{Y})_{\nu+l-l^{(1)}}.\]
\end{lemma}
\begin{proof}
Denote $l \,\coloneqq\, (l_1\ddd l_N)$ and we suppose $|l|=n-k_1-k_2$.
For each $i$, denote by $\mc{Z}_i$ the intersection of $l_i$
generic hyperplanes in $\Pj^{n_i}$. Then $\mc{Z}_i\times \mathbb{P}^{\nu_{-i}}$
has a natural embedding in  $\mathbb{P}^{\nu}$.
Let $\mc{Z}$ be the intersection of the $N$ embedded varieties $\mc{Z}_i\times \mathbb{P}^{\nu_{-i}}$ in $\mathbb{P}^{\nu}$, and
denote $\hat{\mc{X}} \,\coloneqq\, \mc{X}\cap\mc{Z}$.
Since $\mc{Z}$ is defined by generic hyperplanes,
we have $\dim \hat{\mc{X}}=k_2$,
and $\hat{\mc{X}}$ intersects $\mc{Y}$ transversely.
By (\ref{eq:multibezout}), we have
\be\label{eq:hatXcapY}
\deg \hat{\mc{X}}\cap\mc{Y}=\sum_{l^{(2)}\in[k_2]^{(\nu)}}(\deg \hat{\mc{X}})_{l^{(2)}}\cdot(\deg \mc{Y})_{\nu-l^{(2)}}.\ee
Note that $\hat{\mc{X}}\cap \mc{Y}=\mc{X}\cap\mc{Y}\cap\mc{Z}$.
The left hand side of (\ref{eq:hatXcapY}) equals $(\deg \mc{X}\cap\mc{Y})_{l}$.
By definition, for any fixed $l^{(2)}\in[k_2]^{(\nu)}$,  we have
\[(\deg \hat{\mc{X}})_{l^{(2)}}=(\deg \mc{X}\cap\mc{Z})_{l^{(2)}}=(\deg \mc{X})_{l+l^{(2)}}.\]
Therefore,
we conclude this lemma by letting $l^{(1)} \,\coloneqq\, l^{(2)}+l$.
\end{proof}

\section{Finiteness of the Fritz-John Variety}
\label{sc:finite}

The set of complex points $x$ satisfying \reff{eq:jacdef} for all $i\in[N]$
is called the Fritz-John variety.
In this section, we prove that the Fritz-John variety is
finite when the polynomials are generic.

Suppose $E_i$ is the active labeling set for
the $i$th player's optimization problem $\mbF_i(\xmi)$,
which is parameterized by $\xmi = (x_1, \ldots, x_{i-1}, x_{i+1}, \ldots, x_N)$.
Note that $\mc{E}_{i,1} \subseteq E_i \subseteq \mc{E}_{i,1} \cup \mc{E}_{i,2}$.
For the convenience of discussion, we write that
\[
E_i \,= \, \{1, \ldots, m_i \}
\]
for the rest of this section.
For each $i = 1, \ldots, N$, denote the affine varieties
\[
\baray{rcl}
\mc{U}_i & \coloneqq & \{x\in\cpx^n:g_{i,j}(x)=0 \, (j \in [m_i] ) \}, \\
\mc{U}  & \coloneqq  &  \mc{U}_1\cap \cdots \cap\mc{U}_N.
\earay
\]
Consider the determinantal variety
\[
\mc{V}_i \,\coloneqq\,  \left\{  x \in  \cpx^n :
\rank \bbm  \nabla_{x_i} f_i(x) & \nabla_{x_i} g_{i,1}(x) &
  \cdots  & \nabla_{x_i} g_{i,m_i}(x) \ebm  \le m_i
 \right\}.
\]
For each $i=1,\ldots, N$, denote the intersection
$\mc{W}_i\,\coloneqq\, \mc{U}_i\cap\mc{V}_i$, then
\be \label{set:W}
\mc{W} \, \coloneqq \, \mc{W}_1\cap\dots\cap\mc{W}_N
\ee
is the Fritz-John variety, i.e., the set of all complex Fritz-John points.
We are going to show that $\mc{W}$ is a finite set
when $f_i, g_{i,j}$ are generic polynomials.

We consider the multi-projectivization of affine varieties $\mc{U},\mc{V}_i$ and $\mc{W}$.
Denote
\[
\tilde{x}_i  \,\coloneqq\, (x_{i,0},x_{i,1},\ldots,x_{i,n_i}), \quad
\tilde{x}  \,\coloneqq\,  (\tilde{x}_1, \ldots,  \tilde{x}_N  ) .
\]
Let $d_{i,0,k}$ (resp., $d_{i,j,k}$) be the degree of $f_i$ (resp., $g_{i,j}$) in $x_k$.
The degree tuples
\be \label{deg:dij}
d_{i,0} \,\coloneqq\, (d_{i,0,1},\dots,d_{i,0,N}), \quad
d_{i,j} \,\coloneqq\, (d_{i,j,1},\dots,d_{i,j,N})
\ee
are called the multi-degrees of $f_i$, $g_{i,j}$ respectively.
The multi-homogenization of the $i$th player's objective $f_i(\xpi,\xmi)$ is
\[
\tilde{f}_i(\tilde{x}_i,\tilde{x}_{-i}) \,\coloneqq\,
f_i(x_1/x_{1,0},\dots,x_N/x_{N,0})\cdot
\prod_{k=1}^N ( x_{k,0} )^{d_{i,0,k}} .
\]
The multi-homogenization of $g_{i,j}(\xpi,\xmi)$, denoted as $\tilde{g}_{i,j}(\tilde{x}_i,\tilde{x}_{-i})$, is given similarly.
For the convenience of notation, denote that
\[
\baray{c}
\nu \,\coloneqq\, (n_1,\ldots,n_N), \quad
\nu_{-i} \,\coloneqq\, (n_1,\ldots,n_{i-1},n_{i+1},\ldots, n_N), \\
\mathbb{P}^{\nu_{-i}}=\mathbb{P}^{n_1}\times\dots\times\mathbb{P}^{n_{i-1}}
\times\mathbb{P}^{n_{i+1}}\times\dots\times\mathbb{P}^{n_N}.
\earay
\]
Consider the multi-projective varieties
\be \label{sets:tdU}
\baray{rcl}
\wtU_i  &   \coloneqq  & \{
\td{x}\in
\mathbb{P}^{\nu}:
\tilde{g}_{i,j}(\tilde{x})=0 \ ( j  \in  [m_i] ) \}, \\
\wtU  & \coloneqq  &  \wtU_1 \cap \cdots \cap \wtU_N.
\earay
\ee
When $g_{i,j}$ are generic polynomials in $x$,
the codimension of $\wtU_i$ is $m_i$, by Proposition~\ref{pr:bertini},
and $\wtU$ has the codimension $m_1 + \cdots + m_N$.
In the following, we consider the general case that
\[ m_1 + \cdots + m_N \,< \, n. \]
Otherwise, the variety $\wtU$ is empty or zero-dimensional (i.e., it is finite).
When all $f_i$ and $g_{i,j}$ are generic,
the multi-degree of $\frac{\pt f_i(x)}{\pt{x_{i,k}}}$ (resp., $\frac{\pt g_{i,j}(x)}{\pt{x_{i,k}}}$) equals $d_{i,0}-e_i$ (resp., $d_{i,j}-e_i$).
Here, $e_i$ denotes the vector of all zeros
except the $i$th entry being $1$.
Define the multi-projective variety
\be\label{eq:mcVi}
\widetilde{\mc{V}}_i  \, \coloneqq \,
\left\{ \tilde{x}\in\mathbb{P}^{\nu}  :
\rank\,   \wt{J}_i(\tilde{x}_i, \tilde{x}_{-i})          \le m_i
\right\},
\ee
where $\wt{J}_i(\tilde{x}_i, \tilde{x}_{-i})$
is the homogenized Jacobian matrix:
\[
\wt{J}_i(\tilde{x}_i, \tilde{x}_{-i}) \,  \,\coloneqq\,  \,
\left[\begin{array}{cccc}
\frac{\pt\tilde{f}_i(\tilde{x})}{\pt{x_{i,1}}} & \frac{\pt\tilde{g}_{i,1}(\tilde{x})}{\pt{x_{i,1}}} & \cdots   & \frac{\pt\tilde{g}_{i,m_i}(\tilde{x})}{\pt{x_{i,1}}}  \\
\frac{\pt\tilde{f}_i(\tilde{x})}{\pt{x_{i,2}}} & \frac{\pt\tilde{g}_{i,1}(\tilde{x})}{\pt{x_{i,2}}} & \cdots   & \frac{\pt\tilde{g}_{i,m_i}(\tilde{x})}{\pt{x_{i,2}}}  \\
 \vdots &  \vdots & \ddots &  \vdots \\
\frac{\pt\tilde{f}_i(\tilde{x})}{\pt{x_{i,n_i}}} & \frac{\pt\tilde{g}_{i,1}(\tilde{x})}{\pt{x_{i,n_i}}} & \cdots   & \frac{\pt\tilde{g}_{i,m_i}(\tilde{x})}{\pt{x_{i,n_i}}}  \\
\end{array}
\right] .
\]
For convenience, denote the degrees $b_{i,j}$ for each $j = 1, \ldots, N$ such that
\[
b_{i,j} \,\coloneqq\,
\begin{cases}
d_{i,0,j}+d_{i,1,j}+\dots+d_{i,m_i,j}, &  j \ne i  ,  \\
d_{i,0,i}+d_{i,1,i}+\dots+d_{i,m_i,i} - m_i-1, & j=i  .
\end{cases}
\]
Then one can see that all $(m_i+1)$-by-$(m_i+1)$ minors of
$\wt{J}_i(\tilde{x}_i, \tilde{x}_{-i})$
are multi-homogeneous of the multi-degree $(b_{i,1},\dots,b_{i,N})$.
Let $\wt{J}_i^{\circ}(\tilde{x}_i, \tilde{x}_{-i})$
denote the submatrix of $\wt{J}_i(\tilde{x}_i, \tilde{x}_{-i})$
consisting of the right-hand side $m_i$ columns and let
\[
\baray{rcl}
\widetilde{\mc{V}}_i^{\circ}  &  \coloneqq   &
\{\tilde{x}\in\mathbb{P}^{\nu}:\rank\,
\wt{J}_i^{\circ}(\tilde{x}_i, \tilde{x}_{-i})<m_i\}, \\
\widetilde{\mc{W}}_i & \coloneqq  & \widetilde{\mc{V}}_i\cap\wtU_i,\quad
\widetilde{\mc{W}}_i^\circ \,\coloneqq\, \widetilde{\mc{V}}_i^\circ\cap\wtU_i.
\earay
\]
The $\widetilde{\mc{W}}_i$ and $\widetilde{\mc{W}}_i^\circ$ are multi-projective subvarieties of $\wtU_i$. Then
\be  \label{var:tdW}
 \wtW \,\coloneqq\, \wtW_1\cap \dots\cap\wtW_N
\ee
is a multi-projective variety that contains $\mc{W}$.
When defining polynomials are generic,
the set $\wt{\mc{W}}$ is finite, as shown in the following.

\begin{theorem} \label{tm:genfinite}
For every $i,j$, let $d_{i,j} \,\coloneqq\, (d_{i,j,1},\dots,d_{i,j,N})$
be a nonzero tuple of degrees.
If every $f_{i}$ is a generic polynomial of multi-degree $d_{i,0}$,
and $g_{i,j}$ is a generic polynomial of multi-degree $d_{i,j}$,
then we have:

\rm{(i)} There are finitely many points in the Fritz-John variety,
and hence the GNEP has finitely many Fritz-John points.

\rm{(ii)} The linear independence constraint qualification condition holds for all Fritz-John points, and hence every Fritz-John point is a KKT point.
\end{theorem}
\begin{proof}
(i) First, we show that if $m_i\le n_i$,
then the intersection $\wtW_i \deq\wtU_i\cap\wtV_i$ has the codimension $n_i$.
Consider the projection map
\[
\mf{p}_i: \wtW_i \to\mathbb{P}^{\nu_{-i}},
\quad (\td{x}_i, \td{x}_{-i}) \, \mapsto \, \td{x}_{-i}.
\]
When $g_{i,j}$ are all generic, by \cite[Proposition~17.25]{Harris1992}, we have
\[
\cod \wtU_i=m_i, \quad  \cod \wtV_i\le n_i-m_i.
\]
So, the codimension of $\wtW_i$ is at most $n_i$.
Let $\td{u}_{-i}$ be an arbitrary point in $\mathbb{P}^{\nu_{-i}}$ and let
\[
\mc{Z}_i = \{\tilde{x}\in\mathbb{P}^{\nu}:\tilde{x}_{-i}=\tilde{u}_{-i}\} .
\]
Then, by \cite[Theorem~1.24]{Shafarevich2013}, we know
\[\wtU_i\cap\widetilde{\mc{V}}_i\cap \mc{Z}_i \ne\emptyset .\]
In other words, $ \mf{p}_i(\wtW_i)=\mathbb{P}^{\nu_{-i}}$.
Therefore,
by \cite[Theorem~1.25]{Shafarevich2013}, we have
\[
\dim \wtW_i \, \le \,  \dim\, \mf{p}_i( \wtW_i ) + \dim F
= n-n_i+\dim F
\]
for all $\td{u}_{-i}\in \mathbb{P}^{\nu_{-i}}$ and
for every irreducible component $F$ of the fibre $\mf{p}_i^{-1}(\tilde{u}_{-i})$.
For the given $\tilde{u}_{-i}$, the fibre is
\[
\mf{p}_i^{-1}(\tilde{u}_{-i})
=\left\{\tilde{x}_i\in\mathbb{P}^{n_i}\left|
\begin{array}{c}
\tilde{g}_{i,1}(\tilde{x}_i,\tilde{u}_{-i})=\dots=
\tilde{g}_{i,m_i}(\tilde{x}_i,\tilde{u}_{-i})=0,\\
\rank\,\wt{J}_i(\tilde{x}_i,\tilde{u}_{-i})\le m_i
\end{array}\right.
\right\}.
\]
It is zero-dimensional for generic $\tilde{u}_{-i}\in\mathbb{P}^{\nu_{-i}}$
and for generic polynomials $f_i, g_{i,j}$,
by \cite[Proposition~2.1]{Nie2009algebraic}.
So $\dim  \wtW_i \le n-n_i$, and we conclude
\be\label{eq:dimUV}
\dim  \wtW_i = n-n_i,\quad \cod  \wtW_i = n_i .
\ee

Second, we show that the codimension of
$\wtW_i^{\circ}=\wtU_i \cap  \widetilde{\mc{V}}_i^{\circ}$ equals $n_i+1$.
For the given $\td{u}_{-i}\in \mathbb{P}^{\nu_{-i}}$
we have
\[\wtU_i\cap\widetilde{\mc{V}}_i^{\circ}\cap \mc{Z}_i =
\left\{\tilde{x}_i\in\mathbb{P}^{n_i}\left|
\begin{array}{c}
\tilde{g}_{i,1}(\tilde{x}_i,\tilde{u}_{-i})=\dots=
\tilde{g}_{i,m_i}(\tilde{x}_i,\tilde{u}_{-i})=0,\\
\rank\,\wt{J}_i^{\circ}(x_i,u_{-i})<m_i
\end{array}\right.
\right\}.
\]
By \cite[Proposition~2.1]{Nie2009algebraic},
this intersection is empty when every $g_{i,j}$
is a generic polynomial and $\td{u}_{-i}$ is generic in $\mathbb{P}^{\nu_{-i}}$.
Therefore,
\[  \dim  \, \wtW_i^{\circ} \, < \, n-\dim \mc{Z}_i = n-n_i.\]
So, we know $\wtW_i^{\circ}$ is a proper closed subset of $\wtW_i$.
Moreover, by \cite[Proposition~17.25]{Harris1992},
$\cod \wtW_i^{\circ}\le n_i+1$, so
\[
\cod \wtW_i^{\circ} \,=\, n_i+1 .
\]

Last, we show that $\dim \wtW = 0$.
Let $\lmd_i\deq(\lmd_{i,0},\lmd_{i,1}\ddd\lmd_{i,m_i})\in\cpx^{m_i+1}$
be the vector of variables.
Consider the intersection of hypersurfaces in
$\Pj^{\nu}\times \Pj^{m_i}$ given by
\be\label{eq:mcC}
\td{g}_{i,1}=0\ddd \td{g}_{i,m_i}=0,\ \wt{J}_i\cdot\lmd_i=0.
\ee
Denote this multi-projective variety by $\mc{C}_i$.
Then $\tdx\in \wtW_i$ if and only if $\tdx\in\mf{p}(\mc{C}_i)$,
where $\mf{p}$ is the projection from
$\Pj^{\nu}\times \Pj^{m_i}$ to $\Pj^{\nu}$ such that
$\mf{p}(x,\lmd_i)=x$.
By \cite[Theorem~1.25]{Shafarevich2013},
for all $\tdx\in \wtW_i$ and each irreducible component $G$
of the fibre $\mf{p}^{-1}(\tdx)$,
\[  \dim \mc{C}_i  \le \dim  \wtW_i +  \dim  G,  \]
and the equality above hold when $\tdx$ lies in an open subset of $\wtW_i$.
Note that the $\wtW_i^{\circ}$ is a proper closed subset of $\wtW_i$.
When $\tdx \in \wtW_i\setminus\wtW_i^{\circ}$,
there exists the unique $\lambda_i$ such that $\wt{J}_i\cdot\lmd_i=0$,
since the columns of $\wt{J}_i^{\circ}(\tdx)$ are linear independent.
So, we have $\dim \mc{C}_i=\dim  \wtW_i = n-n_i$,
which implies that $\cod \mc{C}_i=n_i+m_i$,
and hence $\mc{C}_i$ is a complete intersection.

We show the intersection $\bigcap_{i=1}^N  \wtW_i$
is proper by induction. Let $k\le N$ and we suppose
\[
\dim \, \bigcap_{i=1}^{k-1} \wtW_i  \, = \,  n-\sum_{i=1}^{k-1}n_i .
\]
Denote
\[
\hat{\mc{C}}_{k-1}\deq \bigcap_{i=1}^{k-1}
\wtW_i \times \Pj^{m_{k}}\subseteq \Pj^{\nu}\times \Pj^{m_k}.
\]
We have $\cod  \hat{\mc{C}}_{k-1}=\sum_{i=1}^{k-1} n_i.$
For the given polynomials $g_{k,1}\ddd g_{k,m_{k}}$, if we vary the coefficients of $f_i(x)$,
then there exist no fixed points for solutions to $\wt{J}_{k}\cdot\lmd_k$
outside of $\mf{p}^{-1}(\wtW_i^{\circ})$,
and $\wtW_i^{\circ}\subsetneq \wtW_i$.
Note that $\mc{C}_i$ is a complete intersection.
Both $\mc{C}_i$ and its projection $\wtW_i$ are pure dimensional,
and $\wtW_i^{\circ}\subsetneq\mc{Y}$
for all irreducible components $\mc{Y}$ of $\wtW_i$.
By Proposition~\ref{pr:bertini}, for generic $f_i$ and $g_{i,j}$,
we have
\[
\dim \, \mc{C}_{k}\cap \hat{\mc{C}}_{k-1}  \, = \, n-\sum_{i=1}^{k}n_i
\ge \dim \, \mf{p}(\mc{C}_{k}\cap \hat{\mc{C}}_{k-1})
= \dim \, \bigcap_{i=1}^{k} (\wtU_i\cap \widetilde{\mc{V}}_i) .
\]
Therefore, by \cite[Proposition~17.25]{Harris1992},
we have
\[ \dim \,  \bigcap_{i=1}^{k} \wtW_i \, = \,  n-\sum_{i=1}^{k}n_i.\]
From (\ref{eq:dimUV}), it is clear that $\dim  \wtW_1 =n-n_1$.
So we conclude the $\bigcap_{i=1}^N \wtW_i$ is proper by induction.
Note that $n=n_1+\dots+n_N$, so we have
\[
\dim \,  \bigcap_{i=1}^{N} \wtW_i \, = \, n-\sum_{i=1}^{N}n_i=0,
\]
which implies the finiteness of $\wtW$.

(ii) For $x\in\re^{n}$,
if $\rank\,\wt{J}^{\circ}_i(\Pj(x))=m_i$,
then the LICQ for $\mbF_i(\xmi)$ hold at $\xpi$.
Therefore, it suffices to show that
$\widetilde{\mc{W}}_i^{\circ}\cap(\cap_{i\ne j\in[N]}\widetilde{\mc{W}}_j)$
is empty for every $i\in[N]$.
Without loss of generality, we fix $i=1$ and show the emptiness of $\widetilde{\mc{W}}_1^{\circ}\cap\widetilde{\mc{W}}_2\cap\dots\cap \widetilde{\mc{W}}_N$.
In item (i), we showed that $\cod  \wtW_1^{\circ}=n_1+1$.
Besides that, for generic polynomials $f_i$ and $g_{i,j}$, the
\[\mc{C}_{-1} \coloneqq \left\{(\tdx,\lmd_2 \ddd \lmd_N)\left|
\begin{array}{c}
\td{g}_{i,1}(\tdx)=0\ddd \td{g}_{i,m_i}(\tdx)=0,\\
\wt{J}_i(\tdx)\cdot\lmd_i=0\ (i=2\ddd N)
\end{array}\right.\right\}.\]
is a complete intersection of hypersurfaces in $\Pj^{\nu}\times \Pj^{m_2}\times\dots\times\Pj^{m_N}$.
Its codimension equals $n-n_1+\sum_{i=2}^Nm_i$,
and we obtain the emptiness for
$\mc{C}_{-1} \cap \wtW_1^{\circ}$
by varying coefficients of $f_i$ and $g_{i,j}$ with $i>1$, and applying Proposition~\ref{pr:bertini} repeatedly.
Moreover, the emptiness for
$\widetilde{\mc{W}}_1^{\circ}\cap\widetilde{\mc{W}}_2\cap\dots\cap \widetilde{\mc{W}}_N$
follows from the fact that it is the image under the projection mapping
$(\tdx,\lmd_2\ddd \lmd_N)$ to $\tdx$ of
$\mc{C}_{-1}  \cap  \wtW_1^{\circ}$.
\end{proof}

\begin{remark}
By Theorem~3.1 (i), when the GNEP is given by generic polynomials,
all Fritz-John points satisfy the KKT conditions.
Note that every KKT point is a Fritz-John point with $\lmd_{i,0}=1$.
Thus, Theorem~3.1 (i) further implies that for generic GNEPs,
the set of Fritz-John points is equal to the set of KKT points.
\end{remark}

\section{Algebraic degrees of multi-projective determinantal varieties}

\label{sc:MPV}

In this section, we study multi-projective determinantal varieties and their algebraic degrees, which are useful for characterizing the algebraic degree of Fritz-John varieties for GNEPs.

Let $\delta \,\coloneqq\, (\delta_1,\dots,\delta_N)$ be a vector in $\N^N$,
$z_i=(z_{i,1},\dots,z_{i,N})$ be a vector of variables,
and $s$ be an integer such that $s\ge r\coloneqq|\delta|$.
For each $i=1,\dots,s$, define the function
\be \label{eq:mcE}
\mc{B}_{\delta}(z_1,\dots,z_s) \,\coloneqq\,
\sum_{1\le i_1<\dots<i_r\le s}
\mc{A}_{\delta}(z_{i_1},\dots,z_{i_r}).
\ee
Note that $\mc{B}_{\delta}(z_1,\dots,z_s)$ equals the coefficient of $t^{\delta}=t_1^{\delta_1}\dots t_N^{\delta_N}$ in
\[\prod_{i=1}^s(1+t_1z_{i,1}+\dots+t_Nz_{i,N}).\]
This is because the coefficient of $t^{\delta}=t_1^{\delta_1}\dots t_N^{\delta_N}$ in the product above equals the sum of coefficients of the same monomial in
\[
\prod_{j=1}^r(z_{i,1}t_1+\dots+z_{i,N}t_N),\quad 1\le i_1<\dots<i_r\le s.
\]
For convenience, denote
\[\N^{[\delta\times s]} \coloneqq \left\{B\in\N^{N\times s} \left|\
B_{i,1}+B_{i,2}+\dots+B_{i,s}=\delta_i\, (i\in[N])
\right.\ \right\} .
\]
The $j$th column of $B$ is $B_{:,j}$.
We define
\be\label{eq:mcS}
\mc{S}_{\delta}(z_1,\dots,z_s)=\sum_{B\in \N^{[\delta\times s]}}\prod_{j=1}^s\binom{|B_{:,j}|}{B_{:,j}}z_j^{B_{:,j}}.\ee
In the above, the $\binom{|B_{:,j}|}{B_{:,j}}$ is the multi-monomial coefficient
\[
\binom{|B_{:,j}|}{B_{:,j}}=\binom{|B_{:,j}|}{B_{1,j}
\ddd B_{N,j}}=\frac{(\sum_{i=1}^s B_{i,j})!}{B_{1,j}!B_{2,j}!\ldots B_{N,j}!}.
\]
One may check that $\mc{S}_{\delta}(z_1,\dots,z_s)$ equals the coefficient of $t^{\delta}=t_1^{\delta_1}\dots t_N^{\delta_N}$ in
\[
\prod_{i=1}^{s}\frac{1}{1- (\sum_{j=1}^N z_{i,j}t_j) }=
\prod_{i=1}^{s}\left(1+\left(\sum_{j=1}^N z_{i,j}t_j\right)+
\left(\sum_{j=1}^N z_{i,j}t_j\right)^2+\cdots\right).
\]

In the following proposition,
we study the algebraic degree of multi-projective varieties defined
by rank deficiency of generic polynomial matrices.
Our main tool is the Thom-Porteous formula \cite[Theorem~14.4]{Ful},
which considers the degeneracy locus of homomorphisms between vector bundles.
We refer to \cite{Ful} for more details.

\begin{proposition}
\label{pr:matdeg}
Let $d_1\le\dots\le d_s\in\N^N$ be vectors of positive integers,
and let $M$ be an $ r\times s$-matrix whose entries $m_{i,j}$ are generic multi-homogeneous polynomials in $\cpx[\tdx_1,\dots,\tdx_N]$
with multi-degrees $d_j$.
For the given $l\in\N^N$ and $\nu=(n_1\ddd n_N)$,
the polynomials $\mc{B}_{\nu-l}$ and $\mc{S}_{\nu-l}$
are given in (\ref{eq:mcE}) and (\ref{eq:mcS}) respectively.
Consider the multi-projective variety
$$\mc{X}_{p} \,\coloneqq\, \{\tilde{x}\in\Pj^{n_1}\times\dots\times\Pj^{n_N}:\rank\,M(\tilde{x})\le p\}.$$

{\rm(i)} If $r\le s$, then $\dim \mc{X}_{r-1}=n-s+r-1$
and for each $l\in[n-s+r-1]^{(\nu)}$,
\[(\deg \mc{X}_{r-1})_l=\mc{B}_{\nu-l}(d_1,\dots,d_s).\]

{\rm(ii)} If $r> s$, then $\dim \mc{X}_{s-1}=n-r+s-1$
and for each $l\in[n-r+s-1]^{(\nu)}$,
\[(\deg \mc{X}_{s-1})_l=\mc{S}_{\nu-l}(d_1,\dots,d_s).\]

\end{proposition}
\begin{proof}
Define the direct sum of line bundles
\[E \,\coloneqq\,
\underbrace{\mc{O}_{\mathbb{P}^{\nu}}\oplus\dots\oplus\mc{O}_{\mathbb{P}^{\nu}}
}_{r\ \mbox{\small times}}, \quad
F \,\coloneqq\, {\mathcal O}_{\mathbb{P}^{\nu}}(d_1)\oplus {\mathcal O}_{\mathbb{P}^{\nu}}(d_2)\oplus \ldots\oplus {\mathcal O}_{\mathbb{P}^{\nu}}(d_s).
\]
In the above, the $\mathcal{O}_{\mathbb{P}^{\nu}}(d)$
is the line bundle of regular functions whose global sections
are multi-homogeneous polynomials of multi-degree $d$ in $\tilde{x}$ on $\mathbb{P}^{\nu}$,
and $\mathcal{O}_{\mathbb{P}^{\nu}}\deq \mathcal{O}_{\mathbb{P}^{\nu}}(\bf{0})$,
where $\bf{0}$ is the zero vector.
Then, for a given map $\rho: {\mathcal O}_{\mathbb{P}^{\nu}}\to {\mathcal O}_{\mathbb{P}^{\nu}}(d_i)$,
there exists a multi-homogeneous polynomial $p$ of multi-degree $d_i$ such that
\[\rho(\tau)=p\cdot \tau \fall p\in\mc{O}_{\mathbb{P}^{\nu}}.\]
So the map $\tilde{x}\to M(\tilde{x})$
defines a homomorphism $\sigma: E\to F$ of vector bundles such that for all $(\tau_1\ddd\tau_r)\in E$,
\[\sigma(\tau_1\ddd\tau_r)=\left(\sum_{i=1}^r m_{i,1}\tau_i,\
\sum_{i=1}^r m_{i,2}\tau_i \ddd\ \sum_{i=1}^r m_{i,s}\tau_i \right).\]
The $\mc{X}_{p}$ is equivalent to the degeneracy locus
$\{x\in\mathbb{P}^{\nu}:
\rank\,\sigma(\tilde {x})\le p\}$,
whose dimension is $n-(r-p)(s-p)$ from the genericity of entries of $M$.
Let \[\begin{array}{ll}
c(F)= (1+d_{1,1}t_1+..+d_{1,N}t_N)\cdot ...\cdot(1+d_{s,1}t_1+..+d_{s,N}t_N).
\end{array}\]
Denote by $c_i$ the sum of terms in the above product whose total degree in $t$ equals $i$.
Then $c(F)=c_0+c_1+\dots+ c_s$.

(i) Assume $r\leq s$. By Thom-Porteous formula (see Theorem~14.4 and Example~14.4.1 in \cite{Ful}),
the $l$-degree of $\mc{X}_{r-1}$ is the coefficient of $t^{\nu}$ in the product $t^l\cdot c_{s-r+1}(F)$,
where $c_{s-r+1}(F)$ is the sum of all terms in $c(F)$
whose total degree in $t$ equals $s-r+1$. Therefore,
$(\deg \mc{X})_l$ equals the coefficient of $t^{\nu-l}$ in $c_{s-r+1}(F)$,
which coincides with $\mc{B}_{\nu-l}(d_1,\dots,d_s)$.

(ii) When $r>s$ and $p=s-1$, instead of using the Thom-Porteous formula,
we assume (i) and proceed by induction on $k>s.$
Assume there exists a positive integer $k\ge s+1$
such that the formula holds for all $r<k$.
Then, let $M_i$ be the submatrix of $M$ consisting of the first $r-i$ rows,
and $N_j$ be the submatrix consisting of the $r-s$ to $r-j$ rows.
Let $X_i$ be the variety where the matrix $M_i$ has rank less than $s$,
and $Y_j$ be the variety such that $\rank\,N_j\leq s-j$.
For each $i$, $\dim X_i=n-r+i+s-1$, and $\dim Y_{i-1}=n-i$.
Since the intersection $X_i\cap Y_{i-1}$ is transversal,
\[ \dim X_i\cap Y_{i-1} \, = \, n-r+s-1 .\]
Besides that,
for $r_i,s_i\in\N^N$ such that $|r_i|=n-r+i+s-1$ and $|s_i|=n-i$,
\[(\deg X_i)_{r_i}=\mc{S}_{\nu-r_i}(d_1,\dots,d_s),\quad (\deg Y_{i-1})_{s_i}=\mc{B}_{\nu-s_i}(d_1,\dots,d_s).\]
The first equation is obtained by induction, and the second equation follows (i).
By Lemma~\ref{lm:interdeg},
\[\deg (X_i\cap Y_{i-1})_l=\sum_{r_i+s_i=l+\nu}\mc{S}_{\nu-r_i}(d_1,\dots,d_f)\cdot \mc{B}_{\nu-s_i}(d_1,\dots,d_s).\]
Note that
\[\mc{X}_{s-1}=X_1\cap Y_{0}\setminus(X_2\cap Y_{1}\setminus(X_3\cap Y_{2}\setminus\dots\setminus (X_s\cap Y_{s-1})\cdots)).\]
Thus we have
\[(\deg \mc{X}_{s-1})_{l}=\sum_{i=1}^s(-1)^{i-1}\sum_{r_i+s_i=l+\nu}\mc{S}_{\nu-r_i}(d_1,\dots,d_s)\cdot \mc{B}_{\nu-s_i}(d_1,\dots,d_s).\]
Consider the identity
\begin{align*}
1&=\prod_{i=1}^s(1+t_1d_{i,1}+\dots+t_Nd_{i,N})/\prod_{i=1}^s(1+t_1d_{i,1}+\dots+t_Nd_{i,N})\\
 &=\prod\limits_{i=1}^s(1+t_1d_{i,1}+\dots+t_Nd_{i,N})(1-(\sum\limits_{j=1}^N d_{i,j}t_j)+(\sum\limits_{j=1}^N d_{i,j}t_j)^2-\cdots)\\
 &=(1+\sum\limits_{\delta}\mc{B}_{\delta}(d_1,\dots,d_s)t^{\delta})(1+\sum\limits_{\delta}(-1)^{|\delta|}\mc{S}_{\delta}(d_1,\dots,d_s)t^{\delta}).
\end{align*}
By comparing the coefficient of $t^{\nu-l}$, we get
\[\mc{S}_{\nu-l}(d_1,\dots,d_s)=\sum_{i=1}^s(-1)^{i-1}
\sum_{r_i+s_i=l+\nu}\mc{S}_{\nu-r_i}(d_1,\dots,d_s)\cdot
\mc{B}_{\nu-s_i}(d_1,\dots,d_s),\]
which implies that
\[\deg(\mc{X}_{s-1})_l=\mc{S}_{\nu-l}(d_1,\dots,d_s).\]
\end{proof}

\section{Algebraic degrees of GNEPs}
\label{sc:degree}

For a GNEP given by generic polynomials,
its algebraic degree counts the number of complex solutions
to the Fritz-John system \reff{eq:FJwithLM}.
This section gives formulae for algebraic degrees of GNEPs.

\subsection{The case of given active constraints}

In this subsection, we consider the Fritz-John variety for the case that the active label sets $ E_i$ are
\be \label{act:Ki}
 E_i \, = \, \{1, \ldots, m_i \}, \quad i = 1,\ldots, N.
\ee
This is the case when all constraints are given by equalities.
Assume the polynomials $f_i$ and $g_{i,j}$ for the GNEP
are generic for the given degrees.
Recall the multi-projective variety $\wtW$ as in \reff{var:tdW} is zero-dimensional,
by Theorem~\ref{tm:genfinite}. Note that
\[
\wtW\setminus\Pj (\mc{W}) =
\left\{ x\in \wtW:  x_{i,0}= 0\ \mbox{for some}\ i\in[N] \right\}.
\]
For each $i$, the intersection $\wtW\cap\{x\in\mathbb{P}^{\nu}:x_{i,0}=0\}=\emptyset$
for each $i\in[N]$.
This is proved in Theorem~\ref{tm:GNEPdegree}.
Therefore, $\Pj (\mc{W})=\wtW$ and hence $|\mc{W}|=|\wtW|$.

Let $m \,\coloneqq\, m_1+\dots+m_N$,
which is the total number of all active constraints.
Recall that $\nu =(n_1\ddd n_N)$.
For convenience, we denote that
\[
N\cdot \nu \, = \, (N\cdot n_1\ddd N\cdot n_N).
\]
Let $\widetilde{\mc{V}} \,\coloneqq\,
\widetilde{\mc{V}}_1\cup\ldots\cup \widetilde{\mc{V}}_N$,
where $\widetilde{\mc{V}_i}$
is the multi-projective determinantal variety given in (\ref{eq:mcVi}).
Note that $\wtW=\wtU\cap\wtV$. Recall the set of labels
$[s]^{\nu}$ is given in (\ref{eq:[s]}) for nonnegative integers $s$.

\begin{theorem}
\label{tm:GNEPdegree}
For each $i,j$,
let
\[
\hat{d}_{i,j} \,\coloneqq\, \Big(
d_{i,j,1},\dots,d_{i,j,i-1},
\max(d_{i,j,i}-1,0),
d_{i,j,i+1},\dots,d_{i,j,N} \Big).
\]
Let $\nu \,\coloneqq\, (n_1,\dots,n_N)$ and
\be\label{eq:Omega}
\Omega \,\coloneqq\,
\left\{\big(l^{(0)},\dots,l^{(N)}\big)
\left|\begin{array}{c}
|l^{(0)}|\in [n-m]^{(\nu)}, \\
|l^{(i)}|=[n-n_i+m_i]^{(\nu)},\, i\in[N] \\
l^{(0)}+\dots+l^{(N)}=N\cdot\nu
\end{array}\right.\right\}.
\ee
Assume the active label sets are as in \reff{act:Ki},
each $f_i$ is generic of the multi-degree $d_{i,0}$
and each $g_{i,j}$ is generic of the multi-degree $d_{i,j}$.
Then, the number of complex Fritz-John points
is given by the following sum
\be\label{eq:GNEPdeg}
\sum_{(l^{(0)},\dots,l^{(N)})\in\Omega}
\mc{A}_{\nu-l^{(0)}}(d_{1,1},\dots,d_{1,m_i},\dots,d_{N,m_N})\cdot
\prod_{i=1}^N \mc{S}_{\nu-l^{(i)}}(\hat{d}_{i,0},\dots,\hat{d}_{i,m_i}).
\ee
In the above, $\mc{A}_{\nu-l^{(0)}}$ is given as in \reff{eq:mcA}
and $\mc{S}_{\nu-l^{(i)}}$ is given as in \reff{eq:mcS}.
\end{theorem}
\begin{proof}
Note that the intersection $\wtW=\cap_{i=1}^N\wtW_i$ is zero-dimensional.
We consider that each $m_i \le n_i$,  because otherwise
the Fritz-John variety is empty when $g_{i,j}$ are generic.
Let $\lambda_i\deq(\lambda_{i,0}\ddd\lmd_{i,m_i})\in \cpx^{m_i+1}$ be variables.
Consider the multi-projective variety $\mc{C}$ given by
\[g_{1,1}(\tdx)=g_{1,2}(\tdx)=\dots=g_{N,m_N}(\tdx)=0,\]
\[ \wt{J}_1(\tdx)\cdot \lmd_1=0\ddd\wt{J}_N(\tdx)\cdot \lmd_N=0.\]
The $\mc{C}$ is defined by vanishing $n+m$ multi-homogeneous polynomials.
So the tuple $(\tdx,\lmd_1\ddd\lmd_N)\in\mc{C}$ if and only if
$x\in \wtW$ and each $\lmd_i$ solves $\wt{J}_i(\tdx)\cdot \lmd_i=0$.
When all $f_i$ and $g_{i,j}$ are generic,
if $\tdx\in\wtW$,
then the $\wt{J}_i(\tdx)\cdot \lmd_i=0$ has a unique solution for all $i$.
This is because the columns of $\wt{J}_i^{\circ}(\tdx)$
are linearly independent if $x\in\wtW$,
by Theorem~\ref{tm:genfinite}.
Therefore, $\dim \mc{C}=0$ and $\mc{C}$ is a complete intersection of hypersurfaces.
We note that for the given constraining polynomials $g_{1,1},g_{1,2}\ddd g_{N,m_N}$,
these hypersurfaces intersect without any fixed point
when we vary coefficients of $f_1\ddd f_N$.
By Proposition~\ref{pr:bertini}, these hypersurfaces intersect transversely,
and $\mc{C}$ is smooth.
Let $\pi$ be the projection that maps $(\tdx,\lmd_1\ddd\lmd_N)$ to $\tdx$.
Since the fibers are linear, the finite set of points in $\mc{C}$ is mapped bijectively to  $\mc{U}\cap\mc{V}_1\cap\dots\cap\mc{V}_N$.
In particular, $\mc{U},\mc{V}_1\ddd\mc{V}_N$ intersect transversely
and we may evaluate $\deg \wtW$ using Lemma~\ref{lm:interdeg}.
Moreover, $\mc{C}\cap\{x\in\mathbb{P}^{\nu}:x_{i,0}=0\}=\emptyset$.
To see this, we assume that $i=1$, without loss of generality.
Then, the $\mc{C}\cap \{x\in\mathbb{P}^{\nu}:x_{1,0}=0\}$
is the intersection of hypersurfaces in
$\Pj^{n_1-1}\times \Pj^{n_2}\times\dots\times\Pj^{n_N}$
that are defined by coefficients in $x$ of
\[
f^{hom}_i(x) \coloneqq  \td{f}_i(\hat{x}),\quad
g^{hom}_{i,1}(\hat{x}) \coloneqq \td{g}_{i,1}(\hat{x}),  \, \ldots,  \,
g^{hom}_{i,m_i}(x) \coloneqq \td{g}_{i,m_i}(\hat{x}),
\]
where $\hat{x}\deq(0,x_{1,1}\ddd x_{1,n_1},x_{2,0},x_{2,1}\ddd x_{N,n_N})$.
Then, the emptiness of $\mc{C}\cap \{x\in\mathbb{P}^{\nu}:x_{1,0}=0\}$
follows from Proposition~\ref{pr:bertini} and Theorem~\ref{tm:genfinite}.
Therefore, the $\wtW$, as the image of $\mc{C}$ under the projection
which maps $(\tdx,\lmd_1\ddd\lmd_N)$ to $\tdx$,
is disjoint with $\{x\in\mathbb{P}^{\nu}:x_{1,0}=0\}$.

By Proposition~\ref{pr:matdeg}, for a vector $l\in[n-n_i+m_i]^{(\nu)}$, we have
\[
(\deg \mc{V}_i)_l = \mc{S}_{\nu-l}(\hat{d}_{i,0},\dots,\hat{d}_{i,m_i}).
\]
Consider the intersection
\[\widetilde{\mc{V}}=\widetilde{\mc{V}}_1\cap\widetilde{\mc{V}}_2\cap \dots\cap \widetilde{\mc{V}}_N=(\dots((\widetilde{\mc{V}}_1\cap\widetilde{\mc{V}}_2)\cap
\widetilde{\mc{V}}_3)\cap \dots\cap \widetilde{\mc{V}}_N).\]
This intersection is transversal
and $\dim \widetilde{\mc{V}}=m-n.$

Now we show (\ref{eq:GNEPdeg}) by induction. For each $i\in[N]$,
denote $s_i\deq n-n_i+m_i$.
Suppose
\be\label{eq:inductionas}
(\deg \widetilde{\mc{V}}_1\cap\widetilde{\mc{V}}_2
\cap \dots\cap \widetilde{\mc{V}}_k)_l
=\sum_{\substack{l^{(i)}\in[s_i]^{(\nu)}\,(i\in[k])\\l^{(1)}+\dots+l^{(k)}=(k-1)\cdot\nu+l}}
\prod_{i=1}^k\mc{S}_{\nu-l^{(i)}}(\hat{d}_{i,0},\dots,\hat{d}_{i,m_i})
\ee
holds for some $k\le N$ and for each $l\in[n-\sum_{i=1}^k(n_i-m_i)]^{(\nu)}$.
Then, by Lemma~\ref{lm:interdeg},
\[
 (\deg \widetilde{\mc{V}}_1\cap\widetilde{\mc{V}}_2\cap \dots
 \cap \widetilde{\mc{V}}_{k+1})_l =
\]
\[
\sum\limits_{\substack{l^{(k+1)}\in[s_{i+1}]^{(\nu)}\\ l^{(k+1)}\ge l\\}}
(\deg \widetilde{\mc{V}}_1\cap\widetilde{\mc{V}}_2\cap \dots\cap \widetilde{\mc{V}}_k)_{\nu+l-l^{(k+1)}}\cdot
\mc{S}_{\nu-l^{(k+1)}}(\hat{d}_{i,0},\dots,\hat{d}_{i,m_i}).
\]
For the right-hand side, we also have
\begin{align*}
&\sum\limits_{\substack{l^{(k+1)}\in[s_{i+1}]^{(\nu)}\\ l^{(k+1)}\ge l\\}}
(\deg \widetilde{\mc{V}}_1\cap\widetilde{\mc{V}}_2\cap \dots\cap \widetilde{\mc{V}}_k)_{\nu+l-l^{(k+1)}}\cdot\mc{S}_{\nu-l^{(k+1)}}(\hat{d}_{i,0},\dots,\hat{d}_{i,m_i})\\
=\quad&\sum_{\substack{l^{(k+1)}\in[s_{i+1}]^{(\nu)}\\ l^{(k+1)}\ge l\\}}\mc{S}_{\nu-l^{(k+1)}}(\hat{d}_{i,0},\dots,\hat{d}_{i,m_i})\sum_{(l^{(1)},\dots,l^{(k)})\in\hat{\Omega}}
\prod_{i=1}^k\mc{S}_{\nu-l^{(i)}}(\hat{d}_{i,0},\dots,\hat{d}_{i,m_i})\\
=\quad&\sum_{\substack{l^{(i)}=[s_i]^{(\nu)}\, (i\in[k+1])\\l^{(1)}+\dots+l^{(k+1)}=k\cdot\nu+l}}\prod_{i=1}^{k+1}\mc{S}_{\nu-l^{(i)}}(\hat{d}_{i,0},\dots,\hat{d}_{i,m_i}).
\end{align*}
In the above, the \[\hat{\Omega}\deq
\left\{\big(l^{(1)},\dots,l^{(k)}\big)
\left|\begin{array}{c}
l^{(i)}=[n-n_i+m_i]^{(\nu)},\, i\in[k] \\
l^{(1)}+\dots+l^{(k)}=k\cdot\nu+l-l^{(k+1)}
\end{array}\right.\right\}.\]
Clearly, \reff{eq:inductionas} holds when $k=1$,
by Proposition~\ref{pr:matdeg}.
Therefore, we have
\[
(\deg \wtV)_l \,\, =\sum_{\substack{l^{(i)}\in[s_i]^{(\nu)}\, (i\in[N])\\l^{(1)}+\dots+l^{(N)}=(N-1)\cdot\nu+l}}
\prod_{i=1}^N\mc{S}_{\nu-l^{(i)}}(\hat{d}_{i,0},\dots,\hat{d}_{i,m_i})
\]
for all $l=[m]^{(\nu)}$, by induction.
Note that the $\wtU$ is given by vanishing $\td{g}_{i,j}$ for all $i\in[N]$ and $j\in[m_i]$,
and all $g_{i,j}$ are generic.
For every $l\in[n-m]^{(\nu)}$,
we have
\[(\deg{\wtU})_l=\mc{A}_{\nu-l}(d_{1,1},\dots,d_{1,m_i},\dots,d_{N,m_N}).\]
Since $\dim \widetilde{\mc{V}}\cap\wtU=0$, by Lemma~\ref{lm:interdeg},
the degree of $\widetilde{\mc{V}}\cap\wtU$ is
\[\begin{gathered}
\sum_{l^{(0)\in[n-m]^{(\nu)}}}(\deg \wtU)_{l^{(0)}}\cdot(\deg \widetilde{\mc{V}})_{\nu-l^{(0)}} = \\ 
\sum_{(l^{(0)},\dots,l^{(N)})\in\Omega}
\mc{A}_{\nu-l^{(0)}}(d_{1,1},\dots,d_{1,m_i},\dots,d_{N,m_N})
\cdot\prod_{i=1}^N\mc{S}_{\nu-l^{(i)}}(\hat{d}_{i,0},\dots,\hat{d}_{i,m_i}).
\end{gathered}
\]
\end{proof}

\subsection{The case without given active constraints}

To count the number of all Fritz-John points,
we need to consider all possibilities of active label sets.
For the $i$th player's optimization problem,
let $ E_i$ be the active label set.
Then, it is clear that
$ \mc{E}_{i,1}\subseteq E_i\subseteq\mc{E}_{i,1}\cup\mc{E}_{i,2} $.
By Theorem~\ref{tm:GNEPdegree},
if we enumerate $ E_i$ by $[m_i]$ for each $i$,
and let $E \,\coloneqq\,(E_1\ddd  E_N)$,
then the algebraic degree for the Fritz-John variety with
the tuple of active labeling sets $E$
is given by (\ref{eq:GNEPdeg}).
Denote the set of all possible active sets
\be\label{eq:Lambda}
\Lambda \coloneqq \left\{(E_1\ddd E_N)\left|\begin{array}{c}
\mc{E}_{i,1}\subseteq E_i\subseteq\mc{E}_{i,1}\cup\mc{E}_{i,2}\ (i\in[N]),\\
|E_i|\le n_i
\end{array}\right. \right\}. \ee
Then, $\Lambda$ is the set of all possible active label sets
for the GNEP when every $g_{i,j}$ is generic.

\begin{theorem}
\label{tm:GNEPdegreeinactive}
For each $i\in[N]$ and $j\in\mc{E}_{i,1}\cup\mc{E}_{i,2}$, let
$d_{i,j} \,\coloneqq\, (d_{i,j,1},\dots,d_{i,j,N})$
be a tuple of degrees, and let
\[
\hat{d}_{i,j} \, \coloneqq  \, \Big(d_{i,j,1}, \,\ldots, \, d_{i,j,i-1}, \,
\max(d_{i,j,i}-1,0), \, d_{i,j,i+1}, \, \ldots, \, d_{i,j,N} \Big).
\]
Assume each $f_i$ is generic of the multi-degree $d_{i,0}$
and each $g_{i,j}$ is generic of the multi-degree $d_{i,j}$.
Let $\Lambda$ be the set of active label sets given in (\ref{eq:Lambda}).
For each $E\deq(E_1\ddd E_N)\in\Lambda$,
denote by $\mc{D}_E$ the algebraic degree given in (\ref{eq:GNEPdeg}) with every $ E_i$ enumerated as $[m_i]$.
Then, the total number of all complex Fritz-John points is 
\be\label{eq:GNEPdeginactive}
\sum_{I\in\Lambda}\mc{D}_E.
\ee
\end{theorem}
\begin{proof}
For the $E\in\Lambda$,
we denote by $\mc{K}_{E}$ the set of complex Fritz-John points with active constraint labels $E$.
Then, it suffices to show that for arbitrary two distinct $E^{(1)}, E^{(2)}\in\Lambda$,
the $\mc{K}_{ E^{(1)}}\cap \mc{K}_{ E^{(2)}}=\emptyset$.

Since $ E^{(1)}\ne  E^{(2)}$,
there exists $i\in[N]$ and $j\in E^{(2)}_i$ such that $j\notin  E^{(1)}$.
By Theorem~\ref{tm:genfinite},
the $\mc{K}_{ E^{(1)}}$ is a finite set.
Then, the equations
\[\td{g}_{i,j}(\tdu)=0\quad (\tdu\in\mc{K}_{ E^{(1)}})\]
define a proper close subset of the space of coefficients for $g_{i,j}$.
This implies that
\[ \mc{K}_{ E^{(1)}}\cap \{ \tdx\in\Pj^{\nu}: \td{g}_{i,j}(\tdx)=0 \}=\emptyset \]
when the $g_{i,j}$ is generic in $\cpx[x]_{d_{i,j}}$.
Note that all points in $\mc{K}_{ E^{(2)}}$ vanish $g_{i,j}$ since $j\in E^{(2)}$.
We know $\mc{K}_{ E^{(1)}}\cap \mc{K}_{ E^{(2)}}=\emptyset$.
\end{proof}

\subsection{The case of NEPs}

Now we consider the special case that the GNEP is an NEP.
Recall that the GNEP reduces to an NEP if for every $i\in[N]$,
all $g_{i,j}$ are independent of $\xmi$.
For the NEP, the multi-degree for $g_{i,j}(\xpi)$ is $(0\ddd 0, d_{i,j}, 0 \ddd 0 )$,
that the $i$th entry $d_{i,j}$ is the degree of $g_{i,j}(\xpi)$ in the variable of $\xpi$,
and all other entries are zero.
Moreover, if $g_{i,j}(\xpi)$ is generic in $\cpx[\xpi]_d$,
then it is a generic polynomial in $\cpx[x]$ whose multi-degree is $d\cdot e_i$.
For the given nonnegative integer $m$ and $i\in[N]$,
let $z_0\,\coloneqq\,(z_{0,1}\ddd z_{0,N})$ and $z_1\ddd z_{m}$ be variables,
and let $\delta\,\coloneqq\,(\delta_1\ddd \delta_N)$ be degrees,
we define
\be \label{Ti:dt}
\begin{gathered}
\mc{T}^{(i)}_{\delta}(z_0,z_1\ddd z_m)\,\coloneqq\,\\
\sum_{\substack{(\eta_0\ddd \eta_m)\in \N^{m+1}  \\ \eta_0+\ldots+\eta_m=\delta_i}}
z_{0,i}^{\eta_0} z_1^{\eta_1}\ldots z_m^{\eta_m}\cdot(\eta_0+
\sum_{i \ne j\in[N] }\delta_j)!
\prod_{ i \ne j\in[N] }
\frac{ z_{0,j}^{\delta_j}}{\dt_j!} .
\end{gathered}
\ee

\begin{theorem}
\label{tm:NEPdegree}
For each $i,j$, let
$d_{i,0} \,\coloneqq\, (d_{i,0,1},\dots,d_{i,0,N})$
be a tuple of multi-degrees,
and let $d_{i,j}$ be a nonnegative integer.
Denote $\nu \,\coloneqq\, (n_1,\dots,n_N)$.
For each $i\in[N]$ and $j\in\mc{E}_{i,1}\cup\mc{E}_{i,2}$,
suppose $f_i$ is a generic polynomial in $x$ with multi-degree $d_{i,0}$,
and $g_{i,j}$ is a generic polynomial in $\xpi$ with degree $d_{i,j}$.

{\rm(i)} For the given tuple of active label sets $E\,\coloneqq\,(E_1\ddd E_N)$,
if we enumerate each $E_i$ as $[m_i]$, and let
\[
\Theta \,\coloneqq\,
\left\{\big(l^{(1)},\dots,l^{(N)}\big)
\left|\begin{array}{c}
l^{(i)} \in [n-n_i+m_i]^{\nu},\, i=1, \ldots, N, \\
l^{(1)}+\dots+l^{(N)}=(N-1)\cdot\nu + (m_1\ddd m_N)
\end{array}\right.\right\},
\]
then the number of complex Fritz-John points is
\be\label{eq:NEPdegree}
\sum_{(l^{(1)},\dots,l^{(N)})\in \Theta }
\prod_{i=1}^N d_{i,1}\cdot \ldots\cdot d_{i,m_i} \cdot
\mc{T}_{\nu-l^{(i)}}^{(i)}(d_{i,0}-e_i, d_{i,1}-1\ddd d_{i,m_i}-1).
\ee

{\rm(ii)} Under the same settings as in (i),
denote by $\mc{D}_{E}$ the algebraic degree given in (\ref{eq:NEPdegree}),
and let $\Lambda$ be the set of all possible active tuples as in (\ref{eq:Lambda}).
Then, the total number of all complex Fritz-John points is
\[
\sum_{E\in\Lambda}\mc{D}_{E}.
\]

\end{theorem}
\begin{proof}
(i) Since each $g_{i,j}$ is generic in $\cpx[\xpi]_{d_{i,j}}$,
so it is generic in the space of all polynomials in $\cpx[x]$ with multi-degree
$h_{i,j}\,\coloneqq\,e_i\cdot d_{i,j}$.
Therefore, we may apply Theorem~\ref{tm:GNEPdegree} to compute the algebraic degree for the case that every player's active label set is $[m_i]$.

Let $m\,\coloneqq\, m_1+\ldots+m_N$.
For each $\delta\,\coloneqq\,(\delta_1\ddd \delta_N)$ such that $\delta_1+\ldots+\delta_N=m$,
the $\mc{A}_{\dt}(h_{1,1},\dots,h_{1,m_i},\dots,h_{N,m_N})$ is the coefficient of $t^{\dt}$ in the product
\[\prod_{i=1}^N\prod_{j=1}^{m_i}(h_{i,j,1}t_1+\ldots+h_{i,j,N}t_N).\]
Therefore, $\mc{A}_{\dt}(h_{1,1},\dots,h_{1,m_i},\dots,h_{N,m_N})$ is nonzero if and only if $\delta=(m_1\ddd m_N)$,
and \[\mc{A}_{(m_1\ddd m_N)}(h_{1,1},\dots,h_{1,m_i},\dots,h_{N,m_N})=\prod_{i=1}^N d_{i,1}\cdot \ldots\cdot d_{i,m_i}.\]

For a given $i\in[N]$, consider
\[\mc{S}_{\delta}(\hat{d}_{i,0}, \hat{h}_{i,1},\dots,\hat{h}_{i,m_i})=\sum_{B\in \N^{[\delta\times (m_i+1)]}}\binom{|B_{:,1}|}{B_{:,1}}\hat{d}_{i,0}^{B_{:,1}}\cdot\prod_{j=1}^{m_i}\binom{|B_{:,j+1}|}{B_{:,j+1}}\hat{h}_{i,j}^{B_{:,j+1}},\]
where
\[\hat{d}_{i,0}\,\coloneqq\, d_{i,0}-e_i,\quad \hat{h}_{i,j}\,\coloneqq\, h_{i,j}-e_i.  \]
For the $B\in \N^{[\delta\times (m_i+1)]}$,
if there exists $(r,s)$ such that $r\ge 2$, $s\ne i$, and $B_{r,s}\ne0$,
then $\prod_{j=1}^{m_i}\binom{|B_{:,j+1}|}{B_{:,j+1}}\hat{h}_{i,j}^{B_{:,j+1}}=0$ since $\hat{h}_{i,r-1,s}=0$.
On the other hand, if we assume $B_{r,s}=0$ for all $r\ge 2$ and $s\ne i$,
then $B_{1,s}=\delta_s$ for all $s=i$.
This is because the sum for the $s$th row of $B$ equals $\delta_s$.
If we denote the $i$th row of $B$ by $(\eta_0\ddd \eta_{m_i})$,
then $\eta_0+\eta_1+\ldots+\eta_{m_i}=\delta_i$.
By considering all possible $B\in \N^{[\delta\times (m_i+1)]}$, we get
\[
\mc{S}_{\delta}(\hat{d}_{i,0}, \hat{h}_{i,1},\dots,\hat{h}_{i,m_i})
= \mc{T}_{\delta}^{(i)}(d_{i,0}-e_i, d_{i,1}-1\ddd d_{i,m_i}-1).
\]
We get the formula (\ref{eq:NEPdegree}) by Theorem~\ref{tm:GNEPdegree}.

(ii) This is implied by Theorem~\ref{tm:GNEPdegreeinactive}
and the item (i).
\end{proof}

\subsection{The case of non-generic polynomials}

When the polynomials $f_i, g_{i,j}$ are not generic,
the set $\mc{W}$ may not be finite.
Even if it is a finite set,
the degree formulae in Theorem~\ref{tm:GNEPdegree} and Theorem~\ref{tm:GNEPdegreeinactive} only give an upper bound for the number of points in $\mc{W}$, and it may not be sharp.
For such cases, we exploit the perturbation argument to
get a better upper bound for the algebraic degree.
For the given tuple $E\,\coloneqq\,(E_1\ddd E_N)$
of active constraint label sets,
we enumerate $ E_i$ as $[m_i]$.
Denote by $\check{d}_{i,j}^{(k)}$ (resp., $\check{d}_{i,0}^{(k)}$)
the multi-degree for the $k$th entry of $\nabla_{x_i}g_{i,j}$ (resp., $\nabla_{x_i}f_i$).
For every $i=1\ddd N$ and $j=0\ddd m_i$, we let
\be   \label{eq:nongdijhat}
\check{d}_{i,j} \,\coloneqq\,
\left(\max_{k\in[n_i]}\check{d}_{i,j,1}^{(k)},
\max_{k\in[n_i]}\check{d}_{i,j,2}^{(k)} \ddd
\max_{k\in[n_i]}\check{d}_{i,j,N}^{(k)}\right).
\ee
Recall the label set $\Omega$ in (\ref{eq:Omega}).

\begin{theorem}  \label{tm:nong}
For the GNEP, let $d_{i,0}$ be the multi-degrees of $f_i$,
$d_{i,j}$ be the multi-degrees of $g_{i,j}$,
and for each $i,j$, the $\check{d}_{i,j}$ is given by (\ref{eq:nongdijhat}).
Given the tuple of active constraint label set $E=([m_1] \ddd [m_N])$,
if there are finitely many complex Fritz-John points,
then their total number is bounded by the following sum
\be\label{eq:GNEPdegnong}
\sum_{(l^{(0)},\dots,l^{(N)})\in\Omega}
\mc{A}_{\nu-l^{(0)}}(d_{1,1},\dots,d_{1,m_i},\dots,d_{N,m_N})\cdot
\prod_{i=1}^N \mc{S}_{\nu-l^{(i)}}(\check{d}_{i,0},\dots,\check{d}_{i,m_i}).
\ee
Moreover, if we let $\Lambda$ be the set of all possible tuples of active constraint labels
and denote by $\check{D}_{E}$ the bound in (\ref{eq:GNEPdegnong}) with the tuple of active constraints $E$,
then the number of all Fritz-John points is bounded by $\sum_{E\in\Lambda}\check{D}_{E}$.
\end{theorem}
\begin{proof}
Given the active constraint $E=(E_1\ddd E_m)$ such that every $E_i=[m_i]$.
Let $\nabla_{\xpi} f_{i}+t\phi_i$ with $\phi_i\in \cpx[x]_{\check{d}_{i,0}}$ be a generic perturbation of $\nabla_{\xpi} f_{i}$ parameterized by $t$.
Similarly, for each $j\in[m_i]$,
we let $\nabla_{x_i}g_{i,j}+ t\gamma_{i,j}$ with $\gamma_{i,j} \in(\cpx[x]_{\check{d}_{i,j}})^{m_i}$ be the generic perturbation of $\nabla_{x_i}g_{i,j}$.
Denote the multi-homogenization in $x$ of $\nabla_{x_i}g_{i,j}+ t\gamma_{i,j}$ by $(\widetilde{g_{t}})_{i,j}$,
and the multi-homogenization in $x$ of
$\nabla_{\xpi} f_i+t\phi_i$ by $(\widetilde{f_{t}})_{i}$. Let
\[(J_{t})_i(\tilde{x}) \,\coloneqq\,
\bbm
(\widetilde{f_{t}})_{i}(\tilde{x}) &(\widetilde{g_{t}})_{i,1}(\tilde{x}) & \dots &(\widetilde{g_{t}})_{i,m_i}(\tilde{x})
\ebm,
\]
\[
(\widetilde{\mc{V}}_{t})_i \,\coloneqq\, \left\{
(\tilde{x}_i,\tilde{x}_{-i},t ) \in\mathbb{P}^{\nu}\times \cpx:
\rank\,(\wt{J}_{t})_i(\tilde{x})\le \eta_i\right\}.
\]
Then
$(\widetilde{\mc{V}}_{t})_i \subseteq \mathbb{P}^{\nu}\times \cpx$
is a variety with a projection $(\widetilde{\mc{V}}_{t})_i\to \cpx$ such that each fiber
$(\widetilde{\mc{V}}_{t=a})_i\subseteq \mathbb{P}^{\nu}$
is a multi-projective variety for each $a\in \cpx$.
Furthermore, since all $\phi_i$ and $\gm_{i,j}$ are generic, the set
\[
\wtW(a)\deq (\wtV_{t=a})_1\cap \dots \cap (\wtV_{t=a})_N\cap \wtU
\]
is finite for all but at most finitely many $a\in \cpx$,
by Theorem~\ref{tm:genfinite}.
In fact, by the genericity assumptions as above,
$(\wtV_t)_1\cap \dots \cap (\wtV_t)_N\cap \wtU$ is $1$-dimensional.
Note that each irreducible component of $(\wtV_t)_1\cap \dots \cap (\wtV_t)_N\cap \wtU$ has a dimension not less than one, by Theorem~\ref{tm:genfinite} (see also \cite[Proposition~17.24]{Harris1992}).
For each $u\in (\wtV_t)_1\cap \dots \cap (\wtV_t)_N\cap \wtU$,
it is contained in an irreducible component whose dimension is greater than or equal to one.
So, there exists a curve $u_t \subset \widetilde{\mc{V}_{t}}_1\cap \dots \cap \widetilde{\mc{V}_{t}}_N\cap \wtU$
that maps onto  $\cpx$ and such that $u\in u_t\cap \widetilde{\mc{V}}_1\cap \dots \cap \widetilde{\mc{V}}_N\cap \wtU$.
By genericity of $(\widetilde{g_{t}})_{i,j}$ and elimination theory (see \cite[Chapter~3]{cox2013book}),
for a generic $a\in\cpx$, the $u(a)\in \wtW(a)$ satisfies
\[
\alpha_{\xi}(a)u_{i,j}(a)^{\xi}+
\alpha_{\xi-1}(a)u_{i,j}(a)^{\xi-1}+\dots+
\alpha_{1}(a)u_{i,j}(a)^{1}+\alpha_0(a)=0,
\]
where
\[
\xi=\sum_{(l^{(0)},\dots,l^{(N)})\in\Omega}
\mc{A}_{\nu-l^{(0)}}(d_{1,1},\dots,d_{1,m_i},\dots,d_{N,m_N})\cdot
\prod_{i=1}^N \mc{S}_{\nu-l^{(i)}}(\check{d}_{i,0},\dots,\check{d}_{i,m_i})
\]
and $\alpha_{\xi}(a)$ are rational functions of the coefficients of
$f_i$, $\phi_i$, $g_{i,j}$ and $\gm_{i,j}$.
So there are at most $\xi$ many such curves.
Therefore, the intersection
$\widetilde{\mc{V}}_1\cap \dots \cap \widetilde{\mc{V}}_N\cap \wtU$
has at most $\xi$ points, and (\ref{eq:GNEPdegnong})
gives an upper bound for the algebraic degree of the GNEP.

If for each $ E \in\Lambda$, there exist finitely
many Fritz-John points with active constraints $ E $,
then the set of all complex Fritz-John points is the union over all
$ E \in\Lambda$, and the total number for them is bounded
by the sum of all $\check{\mc{D}}_{E}$.
\end{proof}

\begin{remark}
When $f_i,g_{i,j}$ are all generic polynomials, the multi-degrees of their partial gradients equals $\hat{d}_{i,j}$, and the algebraic degree of the GNEP is given by Theorems~\ref{tm:GNEPdegree} and \ref{tm:GNEPdegreeinactive}.
However, for the given polynomials $f_i,g_{i,j}$, the multi-degrees of $\nabla_{x_i}f_i$ and $\nabla_{x_i}g_i$ may be smaller than $\hat{d}_{i,j}$ componentwisely.
Note that even for the non-generic cases where $\check{d}_{i,j}\ne \hat{d}_{i,j}$ for some $i\in[N]$ and $j\in[m_i]$, it is possible that there are finitely many Fritz-John points.
See Example~\ref{eq:ipjsGNEP} for such an exposition.
When there are finitely many complex Fritz-John points and $\check{d}_{i,j}\ne \hat{d}_{i,j}$,
the formula (\ref{eq:GNEPdeg}) in Theorems~\ref{tm:GNEPdegree} (formula (\ref{eq:GNEPdeginactive}) in Theorem~\ref{tm:GNEPdegreeinactive} when there are inequality constraints) still gives an upper bound for the algebraic degree.
However, this upper bound cannot be sharp, and this motivates us to consider calculating the algebraic degree using $\check{d}_{i,j}$ instead of $\hat{d}_{i,j}$ in Theorem~\ref{tm:nong}.
Indeed, if $\check{d}_{i,j}\ne \hat{d}_{i,j}$, then the upper bound given by (\ref{eq:GNEPdeginactive}) is greater than that given by (\ref{eq:GNEPdegnong}) in Theorem~\ref{tm:nong}.
For instance, for the GNEP in Example~\ref{eq:ipjsGNEP},
the upper bound given by (\ref{eq:GNEPdeginactive}) equals $230$, and the upper bound given by (\ref{eq:GNEPdegnong}) is $190$, which is exactly the number of Fritz-John points of the GNEP.
\end{remark}

In particular, suppose for every $i\in[N]$ and for every $j\in[m_i]$,
the multi-degrees of all entries in $\nabla_{\xpi}g_{i,j}(x)$ are equal.
Then the $\Pj(\mc{V}_i)$
equals the rank-deficient variety for
\be\label{eq:mcJ}
\wt{\mc{J}}_i(\tilde{x}_i, \tilde{x}_{-i}) \,  \,\coloneqq\,  \,
\left[\begin{array}{cccc}
\wt{\frac{\pt f_i}{\pt{x_{i,1}}}}(\tilde{x}) & \wt{\frac{\pt g_{i,1}}{\pt{x_{i,1}}}}(\tilde{x}) & \cdots   & \wt{\frac{\pt g_{i,m_i}}{\pt{x_{i,1}}}}(\tilde{x})  \\
\wt{\frac{\pt f_i}{\pt{x_{i,2}}}}(\tilde{x}) & \wt{\frac{\pt g_{i,1}}{\pt{x_{i,2}}}}(\tilde{x}) & \cdots   & \wt{\frac{\pt g_{i,m_i}}{\pt{x_{i,2}}}}(\tilde{x})  \\
 \vdots &  \vdots & \ddots &  \vdots \\
\wt{\frac{\pt f_i}{\pt{x_{i,N}}}}(\tilde{x}) & \wt{\frac{\pt g_{i,1}}{\pt{x_{i,N}}}}(\tilde{x}) & \cdots   & \wt{\frac{\pt g_{i,m_i}}{\pt{x_{i,N}}}}(\tilde{x})  \\
\end{array}
\right].
\ee
In the above, each $\wt{\frac{\pt g_{i,j}}{\pt{x_{i,k}}}}(\tilde{x})$
denotes the multi-homogenization for $\frac{\pt g_{i,j}}{\pt{x_{i,k}}}(\tilde{x})$.
Denote by $\wt{\mc{Y}}_i$ the multi-projective variety
$\{\tdx\in\Pj^{\nu}: \rank\,\wt{\mc{J}}_i(\tilde{x}_i, \tilde{x}_{-i})\le m_i\}$,
and let
$\wt{\mc{J}^{\circ}_i}(\tilde{x}_i, \tilde{x}_{-i})$
be the submatrix of $\wt{\mc{J}}_i(\tilde{x}_i, \tilde{x}_{-i})$
which consists of the right $m_i$ columns.
If we further assume that the intersection $\wtU$ is transversal,
and for each $i\in[N]$, the
\[
\{\rank\,\wt{\mc{J}^{\circ}_i}(\tilde{x}_i, \tilde{x}_{-i})\le m_i-1\}
\cap \wtU=\emptyset,
\]
then the intersection
\[\wt{\mc{Y}}_1\cap \dots \cap \wt{\mc{Y}}_N \cap\wtU\]
is transversal when all $f_i$ are generic.
This is because each $\wt{\mc{Y}}_i$
is linearly parameterized by coefficients of $f_i$,
and the set of fixed points for this section is contained in the intersection
\[
\{\td{x}\in\Pj^{\nu}: \, \rank\,\wt{\mc{J}^{\circ}_i}(\tilde{x}_i, \tilde{x}_{-i})\le m_i-1\}
\cap \wtU,
\]
if we vary all the coefficients of all $f_i.$
In this case, if all $f_i$ are generic,
then the number of Fritz-John points with active constraints labels
$[m_1]\ddd [m_N]$ is given by (\ref{eq:GNEPdegnong}). From the above discussion
and by Theorem~\ref{tm:nong}, we have the following result.

\begin{proposition}
Under the assumptions in Theorem~\ref{tm:nong},
suppose for all $i\in[N]$ and $j\in[m]_i$,
the multi-degrees of all entries in $\nabla_{\xpi}g_i(x)$ are equal,
and hypersurfaces defined by all $\td{g}_{i,j}(\tdx)=0$ intersect transversely.
Denote
\[
\wt{\mc{Y}}_i\,\coloneqq\,\{ \tdx\in\Pj^{\nu}:
\rank\,\wt{\mc{J}}_i(\tilde{x}_i, \tilde{x}_{-i})\le m_i\} .
\]
If the intersection $\wt{\mc{Y}}_1\cap \dots \cap \wt{\mc{Y}}_N \cap\wtU$ is transversal,
then the upper bound given by (\ref{eq:GNEPdegnong}) is sharp.
\end{proposition}

\section{Computational experiments}
\label{sc:ce}

In this section, we present some computational results of complex Fritz-John points for GNEPs.
Recall that the set $\Lambda$ of all possible active sets $ E  \,\coloneqq\, ( E_1 \ddd  E_N)$ for the GNEP is given in (\ref{eq:Lambda}).
For each GNEP in this section and for every $E\in\Lambda$,
we use the symbolic computation software {\tt Macaulay2} \cite{macaulay2}
to evaluate the dimension of the ideal (of the polynomial ring $\mathbb{Q}[x]$)
\be\label{eq:m2_ideals} I_E = \sum_{i=1}^N I_{E_i}.
\ee
In the above, if we let $m_i=|E_i|$ and we enumerate $E_i:= (1\ddd m_i)$ for each $i\in[N]$, then
\[ I_{E_i} \,:=\, \left\{\begin{array}{ll}
\lip g_{i,1}\ddd g_{i,m_i}, \mathfrak{m}_{i,1},\ddd \mathfrak{m}_{i,\ell_i} \rip, & \mbox{ if } m_i< n_i \\
\lip g_{i,1}\ddd g_{i,m_i} \rip, & \mbox{ if } m_i\ge n_i
\end{array}\right.\]
where $\ell = \binom{n_i}{m_i+1}$, and $(\mathfrak{m}_{i,1},\ddd \mathfrak{m}_{i,\ell_i})$ is the tuple of all $(m_i+1)\times (m_i+1)$ minors of
\[ \lvt\  \nabla_{\xpi} f_i\ \nabla_{\xpi} g_{i,1} \ \nabla_{\xpi} g_{i,2} \ \ldots \nabla_{\xpi} g_{i,m_i}\ \rvt. \]
Thus if we let $\hat{m}_{E_i}:= |E_i| $ when $|E_i|\ge n_i$ and let $\hat{m}_{E_i}:= |E_i|+ \binom{n_i}{|E_i|+1}$ when $|E_i|<n_i$,
then $I_E$ is generated by $\hat{m}_{E_1}+\ldots+\hat{m}_{E_N}$ polynomials, and the number of variables is $n$, i.e., the dimension of $x$.
It is clear that $x$ is a complex Fritz-John point with the active set $E$ if and only if $x$ belongs to the complex affine variety $V(I_E)$.
Moreover, if $\dim I_{E} = 0$, then $\mathbb{Q}[x]/{I_E}$ is a finitely dimensional space, and we use the function {\tt degree} to calculate the dimension of $\mathbb{Q}[x]/{I_E}$, which equals the number of complex Fritz-John points with the active set $E$ (counting multiplicities); see \cite[Theorem~2.1.4]{Nie2023moment} and \cite[Proposition~2.1]{Sturmfels2002}.
Last, when there exist inequality constraints, i.e., the cardinality of $\Lambda$ is greater than one, we evaluate the dimension of $I_{E^{(1)}}+I_{E^{(2)}}$ for all $E^{(1)}, E^{(2)}\in \Lambda$ and $E^{(1)}\ne E^{(2)}$.
If $\dim (I_{E^{(1)}}+I_{E^{(2)}})<0$, then $V(I_{E^{(1)}})\cap V(I_{E^{(2)}})=\emptyset$, thus the algebraic degree of $V(I_{E^{(1)}})\cup V(I_{E^{(2)}})$ equals $\deg(V(I_{E^{(1)}})) + \deg(V(I_{E^{(2)}}))$; see \cite{Nie2009algebraic}.
The computation is conducted on a MacBook Pro laptop
with an Apple M3 Pro processor with 12 cores and 18GB of memory,
in the macOS 14.2.1 system.

Aside from the symbolic computation, we also report the upper bound of algebraic degrees obtained from (\ref{eq:GNEPdeg}) and (\ref{eq:GNEPdegnong}).
When inequality constraints exist, we denote by $\mc{D}_{E}$ the generic algebraic degree given by (\ref{eq:GNEPdeg})
for the Fritz-John variety with the tuple of active labeling set $ E $.
When the constraining polynomials $g_{i,j}$ are not generic,
the upper bound for the complex Fritz John points is denoted by 
$\check{\mc{D}}_{E}$, which is given in (\ref{eq:GNEPdegnong}).
We would like to remark that in the following examples,
we use variables $x$, $y$, $z$ to represent $x_1$, $x_2$ and $x_3$ respectively,
for the cleanness of the paper.

\subsection{Unconstrained problems}
For unconstrained problems,
the Fritz-John condition becomes $\nabla_{x_i}f_i(x)=0$ for all $i=1,\dots,N$.
For the generic $f_i$ with multi-degree $d_{i,0}=(d_{i,0,1},\dots,d_{i,0,N})$,
let $\hat{d_{i}}=(d_{i,0,1},\dots,d_{i,0,i}-1,\dots,d_{i,0,N})$.
Then, the upper bound of the algebraic degree for these equations given by Theorem~\ref{tm:GNEPdegree} equals
\be\label{eq:unconsdegree} \mc{A}_{\nu}(\underbrace{\hat{d}_{1},\dots,\hat{d}_{1}}_{n_1\mbox{ times}},\dots,\underbrace{\hat{d}_{N},\dots,\hat{d}_{N}}_{n_N\mbox{ times}}).\ee
Interestingly, one may check that it also equals
\[
\sum_{\substack{\alpha_{1}+\dots+\alpha_{N}=\nu \\ \alpha_i\in [n_i]^{(\nu)}}}\mc{S}_{\alpha_1}(\hat{d}_1)\cdots\dots\cdot\mc{S}_{\alpha_N}(\hat{d}_N)
=\sum_{\substack{\alpha_{1}+\dots+\alpha_{N}=\nu \\\alpha_i\in [n_i]^{(\nu)}}}\hat{d}_1^{\alpha_1}\cdot\dots\cdot \hat{d}_N^{\alpha_N}.\]

\begin{example}
Consider the two-player NEP that each player solves an unconstrained optimization problem.
The objective function for the first player is
\[\begin{array}{ll}
f_1(x,y)=&4x_1^2y_1^2 - x_1^2y_1y_3 - x_1^2y_1 + 4x_1^2y_3^2 - 2x_1^2y_3 - 2x_1x_2y_1y_2+ 5x_1x_2y_1y_3\\
&+ 3x_1x_2y_2^2 - 3x_1x_2y_2y_3 - 2x_1x_2y_2 - 4x_1x_2y_3^2- x_1x_3y_1^2+ 4x_1x_3y_1y_3  \\
&- 2x_1x_3y_1 + 3x_1x_3 + 3x_1y_2^2 + 3x_1y_2y_3 - x_1y_3- 2x_2^2y_1y_2+ 4x_2^2y_1y_3\\
&- 3x_2^2y_2^2 - 2x_2^2y_3^2 + 2x_2x_3y_2y_3+ 3x_2x_3y_2+3x_2x_3y_3^2 + 4x_2y_1^2 \\
& + 5x_2y_1y_3 + 3x_2y_1 + 4x_2y_2y_3 + 3x_2y_2+ 2x_2y_3^2 - 3x_3^2y_1y_2 + 2x_3^2y_2^2\\
&  + 4x_3^2y_2 - 3x_3^2y_3^2 - 2x_3^2y_3- 4x_3y_1^2 - 3x_3y_2y_3 + 4x_3y_3.
\end{array}
\]
For the second player, the objective function is
\[\begin{array}{ll}
f_2(x,y)=&3x_1^2y_1^2 - x_1^2y_1y_2 + 2x_1^2y_3 - 2x_1x_2y_1^2 + 2x_1x_2y_1y_2+ 3x_1x_2y_1y_3 \\
&- x_1x_2y_2y_3+ 5x_1x_3y_1^2 - x_1x_3y_1 + 5x_1x_3y_3^2- 2x_1y_1y_2 - x_1y_2^2 \\
&+ 3x_1y_2 + 3x_2^2y_1^2 + 3x_2^2y_1y_2+ 4x_2^2y_1y_3 + 3x_2^2y_1 - 2x_2^2y_2^2\\
&- 2x_2^2y_2 + 4x_2x_3y_1y_3+ 3x_2x_3y_2 - 2x_2y_1y_3+ 4x_2y_3- 2x_3^2y_1y_2\\
& + 2x_3^2y_1y_3- x_3^2y_2^2- x_3^2y_3^2- 2x_3^2y_3 + 2x_3y_2y_3- 3y_2^2 - 2y_2 - 2y_3.
\end{array}
\]
In this problem, $N=2$ and $\nu=(3,3)$.
The multi-degrees for both $f_1$ and $f_2$ are $(2,2)$,
and $\hat{d}_1=(1,2)$, $\hat{d}_2=(2,1)$.
The upper bound for the algebraic degree given by (\ref{eq:unconsdegree}) is
\[\mc{A}_{(3,3)}(\hat{d}_1,\hat{d}_1,\hat{d}_1,\hat{d}_2,\hat{d}_2,\hat{d}_2)=245.\]
Since there are no inequality constraints,
$\Lambda = \{ E \}$ with $E = (\emptyset, \emptyset)$.
The ideal $I_E$ is generated by six polynomials, that
\[ I_E = \lip \frac{\pt f_1}{\pt x_{1,1}}, \frac{\pt f_1}{\pt x_{1,2}}, \frac{\pt f_1}{\pt x_{1,3}}, \frac{\pt f_2}{\pt x_{2,1}}, \frac{\pt f_2}{\pt x_{2,2}}, \frac{\pt f_2}{\pt x_{2,3}} \rip. \]
There are six variables, and the first three generators have multi-degrees equal to $(1,2)$ while the later three have multi-degrees equal to $(2,1)$.
By the symbolic computation software {\tt Macaulay2}, we find that the Fritz-John variety has exactly $245$ points (counting multiplicities).
\end{example}

\subsection{NEPs with ball constraints}
Consider the two-player NEP such that each player's feasible strategy vector is bounded within the unit ball:
\begin{equation}
\label{eq:sphereNEP}
\begin{array}{cllcl}
    \min\limits_{x \in \re^3 }& f_1(x,y) &\vline&
    \min\limits_{y \in \re^3 }& f_2(x,y) \\
    \st & \Vert x\Vert^2\le 1 ,&\vline& \st &\Vert y\Vert^2\le 1.\\
\end{array}
\end{equation}
In this problem, $\nu=(3,3)$, $m=2$, and
\[d_{1,1}=(2,0),\ d_{2,1}=(0,2),\ \check{d}_{1,1}=(1,0),\ \check{d}_{2,1}=(0,1).\]
There are totally four cases of active constraints:
\be\label{eq:activeball} E^{(1)}=(\emptyset,\emptyset),\  E^{(2)}=(\{1\},\emptyset),\  E^{(3)}=(\emptyset,\{1\}),\  E^{(4)}=(\{1\},\{1\}).\ee
By Theorem~\ref{tm:nong}, an upper bound of algebraic degree for this NEP is given by
\begin{equation}\label{eq:SphNEPdegree}
\check{\mc{D}}_{ E^{(1)}}+\check{\mc{D}}_{ E^{(2)}}+\check{\mc{D}}_{ E^{(3)}}+\check{\mc{D}}_{ E^{(4)}}.
\end{equation}

\begin{example}
\label{ep:ballNEP}
Consider the two-player ball-constrained NEP (\ref{eq:sphereNEP}).
Let the objective functions be
\[\begin{array}{ll}
f_1(x,y)=&2x_2 - 2x_1 + 5x_3 + 2x_1x_3 + 5x_2x_3 - 3x_1y_1 - 2x_2y_1- 3x_1y_3\\
&- 2x_2y_2 - x_3y_1+ 4x_2y_3 + 4x_3y_3 + 4x_1^2y_1 + 2x_1^2y_2- x_2^2y_1\\
&+ 4x_2^2y_3+ 2x_3^2y_2 + 3x_3^2y_3 - x_1^2- x_3^2 + 4x_1x_2y_1+ 2x_1x_2y_2\\
&- x_1x_3y_1- x_1x_2y_3 - 3x_2x_3y_1 - x_1x_3y_3 - x_2x_3y_2 + 2x_2x_3y_3,
\end{array}
\]
\[\begin{array}{ll}
f_2(x,y)=&2x_1y_1 + 2x_1y_2 + 5x_2y_1 - 3x_2y_2 - 4x_3y_1 + 3x_3y_2 - 6x_3y_3- y_1y_2\\
&- x_1y_1^2+ 3x_1y_3^2 + 2x_2y_2^2+ 2x_2y_3^2 + 4x_3y_3^2 + 2y_2^2 + 4x_2y_1y_2\\
&+ 2x_1y_2y_3 - 3x_3y_1y_2 + 2x_2y_2y_3+3x_3y_1y_3 - x_3y_2y_3.
\end{array}
\]
The multi-degrees for objective functions are $d_{1,0}=(2,1)$ and $d_{2,0}=(1,2)$.
We have
\[\check{\mc{D}}_{ E^{(1)}}=20,\quad \check{\mc{D}}_{ E^{(2)}}=\check{\mc{D}}_{ E^{(3)}}=30,\quad \check{\mc{D}}_{ E^{(4)}}=76.\]
Therefore, the upper bound for the algebraic degree of this NEP given
by Theorem~\ref{tm:nong} (and thus by (\ref{eq:SphNEPdegree})) is $156$.
Also, for this GNEP, since $\hat{d}_{i,j} = \check{d}_{i,j}$
holds for all $i=1\ddd 3$ and $j=0,1$, thus the upper bound for the algebraic degree given by (\ref{eq:GNEPdegnong})is equal to that given by (\ref{eq:GNEPdeg}).
For the symbolic computation, we remark that there are $4$ cases of active sets as shown in (\ref{eq:activeball}),
and all $I_{E^{(j)}}\, (j=1\ddd 4)$ has $6$ variables.
The generators and their multi-degrees are presented in the following:
\begin{itemize}
  \item $I_{E^{(1)}} = \lip \frac{\pt f_1}{\pt x_{1,1}}, \frac{\pt f_1}{\pt x_{1,2}}, \frac{\pt f_1}{\pt x_{1,3}}, \frac{\pt f_2}{\pt x_{2,1}}, \frac{\pt f_2}{\pt x_{2,2}}, \frac{\pt f_2}{\pt x_{2,3}} \rip$: all generators has multi-degrees $(1,1)$.

  \item $I_{E^{(2)}} = \lip g_{1,1}, \mathfrak{m}_{1,1},\mathfrak{m}_{1,2},\mathfrak{m}_{1,3},\frac{\pt f_2}{\pt x_{2,1}}, \frac{\pt f_2}{\pt x_{2,2}}, \frac{\pt f_2}{\pt x_{2,3}}  \rip$: $\deg(g_{1,1}) = (2,0)$, $\deg(\mathfrak{m}_{1,j}) = (2,1)$ for all $j=1\ddd 3$, and $\deg (\frac{\pt f_2}{\pt x_{2,j}}) = (1,1)$ for all $j=1\ddd 3$.

  \item $I_{E^{(3)}} = \lip \frac{\pt f_1}{\pt x_{1,1}}, \frac{\pt f_1}{\pt x_{1,2}}, \frac{\pt f_1}{\pt x_{1,3}}, g_{2,1}, \mathfrak{m}_{2,1},\mathfrak{m}_{2,2},\mathfrak{m}_{2,3} \rip$:
  $\deg (\frac{\pt f_1}{\pt x_{1,j}}) = (1,1)$ for all $j=1\ddd 3$,
  $\deg(g_{2,1}) = (0,2)$, and $\deg(\mathfrak{m}_{2,j}) = (1,2)$ for all $j=1\ddd 3$.
  \item $I_{E^{(4)}} = \lip g_{1,1}, g_{2,1}, \mathfrak{m}_{1,1},\mathfrak{m}_{1,2},\mathfrak{m}_{1,3}, \mathfrak{m}_{2,1},\mathfrak{m}_{2,2},\mathfrak{m}_{2,3} \rip$:
  $\deg(g_{1,1}) = (2,0)$, $\deg(g_{2,1}) = (0,2)$,
  and $\deg(\mathfrak{m}_{1,j}) = (2,1)$, $\deg(\mathfrak{m}_{2,j}) = (1,2)$ for all $j=1\ddd 3$.
\end{itemize}
Using symbolic computation software {\tt Macaulay2},
we find that $V(I_{E^{(1)}})$ has $20$ points,
$V(I_{E^{(2)}})$ and $V(I_{E^{(3)}})$ has $30$ points,
and $V(I_{E^{(4)}})$ has $76$ points (all counting multiplicities).
Besides that, all these affine varieties are pairwisely disjoint.
Thus, the Fritz-John variety has exactly $156$ points (counting multiplicities).
\end{example}

\subsection{GNEP with generic joint-linear constraints}
Consider the NEP such that the $i$th player solves the following optimization problem
\begin{equation}
\label{eq:genericGNEP}
\begin{array}{cllcl}
    \min\limits_{x \in \re^3 }& f_1(x,y) &\vline&
    \min\limits_{y \in \re^3 }& f_2(x,y) \\
    \st & g_1(x,y)\ge0 ,&\vline& \st & g_2(x,y)\ge0.\\
\end{array}
\end{equation}
where $g_i(x,y)$ is a polynomial with multi-degree $(1,1)$, i.e.,
$g_i(x)$ is bilinear in $(x,y)$.
For this GNEP, the $\nu=(3,3)$, $m=2$, and
\[d_{1,1}=d_{2,1}=(1,1),\ \hat{d}_{1,1}=(0,1),\ \hat{d}_{2,1}=(1,0).\]
There are totally four cases of active constraints:
\[ E^{(1)}=(\emptyset,\emptyset),\  E^{(2)}=(\{1\},\emptyset),\  E^{(3)}=(\emptyset,\{1\}),\  E^{(4)}=(\{1\},\{1\}).\]
By Theorem~\ref{tm:GNEPdegreeinactive}, the algebraic degree for this GNEP equals
\begin{equation}\label{eq:genericlinearGNEPdegree}
\mc{D}_{ E^{(1)}}+\mc{D}_{ E^{(2)}}+\mc{D}_{ E^{(3)}}+\mc{D}_{ E^{(4)}}.
\end{equation}

\begin{example}
Let $f_1(x,y)$ and $f_2(x,y)$ be given as in Example~\ref{ep:ballNEP}.
Consider the two-player joint linear constrained GNEP (\ref{eq:genericGNEP}), with
\[\begin{array}{ll}
g_1(x,y)=4x_2 - 3x_1 + x_3 + 3y_1 + y_2 + 5y_3 + x_1y_1 - 3x_1y_2+ 4x_2y_1 \\
\qquad\qquad\qquad\qquad + 2x_1y_3 - 6x_2y_2 + 2x_3y_1 + 3x_2y_3 - x_3y_2 - 4x_3y_3 - 2,\\
g_2(x,y)=8x_1 + 3x_2 + 5x_3 - y_1 + y_3 + 7x_1y_1 - 3x_1y_2 - 6x_2y_1\\
\qquad\qquad\qquad\qquad - x_1y_3 + 4x_2y_2 - 2x_3y_1 - 2x_3y_2 + 7x_3y_3 + 3.
\end{array}
\]
For objective functions, the multi-degrees are $d_{1,0}=(2,1)$, $d_{2,0}=(1,2)$.
We have
\[\mc{D}_{ E^{(1)}}=20,\quad \mc{D}_{ E^{(2)}}=\mc{D}_{ E^{(3)}}=34,\quad \mc{D}_{ E^{(4)}}=62.\]
Therefore, the upper bound for the algebraic degree of this GNEP given by (\ref{eq:genericlinearGNEPdegree}) is $150$.
For the symbolic computation, we remark that there are $4$ cases of active sets,
and all $I_{E^{(j)}}\, (j=1\ddd 4)$ has $6$ variables.
The generators and their multi-degrees are presented in the following:
\begin{itemize}
  \item $I_{E^{(1)}} = \lip \frac{\pt f_1}{\pt x_{1,1}}, \frac{\pt f_1}{\pt x_{1,2}}, \frac{\pt f_1}{\pt x_{1,3}}, \frac{\pt f_2}{\pt x_{2,1}}, \frac{\pt f_2}{\pt x_{2,2}}, \frac{\pt f_2}{\pt x_{2,3}} \rip$: all generators has multi-degrees $(1,1)$.

  \item $I_{E^{(2)}} = \lip g_{1,1}, \mathfrak{m}_{1,1},\mathfrak{m}_{1,2},\mathfrak{m}_{1,3},\frac{\pt f_2}{\pt x_{2,1}}, \frac{\pt f_2}{\pt x_{2,2}}, \frac{\pt f_2}{\pt x_{2,3}}  \rip$: $\deg(g_{1,1}) = (1,1)$, $\deg(\mathfrak{m}_{1,j}) = (1,2)$ for all $j=1\ddd 3$, and $\deg (\frac{\pt f_2}{\pt x_{2,j}}) = (1,1)$ for all $j=1\ddd 3$.

  \item $I_{E^{(3)}} = \lip \frac{\pt f_1}{\pt x_{1,1}}, \frac{\pt f_1}{\pt x_{1,2}}, \frac{\pt f_1}{\pt x_{1,3}}, g_{2,1}, \mathfrak{m}_{2,1},\mathfrak{m}_{2,2},\mathfrak{m}_{2,3} \rip$:
  $\deg (\frac{\pt f_1}{\pt x_{1,j}}) = (1,1)$ for all $j=1\ddd 3$,
  $\deg(g_{2,1}) = (1,1)$, and $\deg(\mathfrak{m}_{2,j}) = (2,1)$ for all $j=1\ddd 3$.
  \item $I_{E^{(4)}} = \lip g_{1,1}, g_{2,1}, \mathfrak{m}_{1,1},\mathfrak{m}_{1,2},\mathfrak{m}_{1,3}, \mathfrak{m}_{2,1},\mathfrak{m}_{2,2},\mathfrak{m}_{2,3} \rip$:
  $\deg(g_{1,1}) = \deg(g_{2,1}) = (1,1)$,
  $\deg(\mathfrak{m}_{1,1})=\deg(\mathfrak{m}_{1,2})=\deg(\mathfrak{m}_{1,3}) = (1,2)$, and $\deg(\mathfrak{m}_{2,1}) = \deg(\mathfrak{m}_{2,2}) = \deg(\mathfrak{m}_{2,3}) = (2,1)$.
\end{itemize}
Using symbolic computation software {\tt Macaulay2},
we find that $V(I_{E^{(1)}})$ has $20$ points,
$V(I_{E^{(2)}})$ and $V(I_{E^{(3)}})$ has $34$ points,
and $V(I_{E^{(4)}})$ has $62$ points (all counting multiplicities).
Besides that, all these affine varieties are pairwisely disjoint.
Thus, the Fritz-John variety has exactly $150$ points (counting multiplicities).
\end{example}

\subsection{GNEP with inner product \& ball constraints}
Consider the two-player GNEP
\be\label{eq:ipjsGNEP}
\begin{array}{cllcl}
    \min\limits_{x \in \re^2 }& f_1(x,y) &\vline&
    \min\limits_{y \in \re^2 }& f_2(x,y) \\
    \st & x^T(y+\mathbf{e})=1,&\vline& \st &x^Tx+y^Ty\le2.\\
        & x_1\ge0 ,x_2\ge0; &\vline
\end{array}
\ee
For this GNEP, the $\nu=(2,2)$, $m=4$, and
\[d_{1,1}=(1,1),\ d_{1,2}=d_{1,3}=(1,0),\ d_{2,1}=(2,2).\]
Since the $g_{1,1}(x,y), g_{1,2}(x,y), g_{1,3}(x,y)$ (resp., $g_2(x,y)$) are not generic in $\re[x,y]_{(1,1)}$ (resp., $\re[x,y]_{(2,2)}$), the multi-degrees for their partial gradients are
\[\check{d}_{1,1}=(0,1),\ \check{d}_{1,2}=\check{d}_{1,3}=(0,0),\ \check{d}_{2,1}=(0,1).\]
Suppose the multi-degrees for $f_1$ and $f_2$ are $d_{1,0}$ and $d_{2,0}$.
Let $\check{d}_{1,0}$ and $\check{d}_{2,0}$ be defined as in Theorem~\ref{tm:nong}.
Note that in the first player's subproblem, three constraints cannot be active simultaneously,
since the last two constraints being active implies that $x_1=x_2=0$, and thus $x^T(y+\mathbf{e}) = 0\ne 1$.
So, there are totally six possibilities of active constraints, which are:
\[\begin{array}{lll}
 E^{(1)}=(\{1\},\emptyset), &  E^{(2)}=(\{1\},\{1\}), &  E^{(3)}=(\{1,2\},\emptyset),\\
 E^{(4)}=(\{1,2\},\{1\}), &  E^{(5)}=(\{1,3\},\emptyset), &  E^{(6)}=(\{1,3\},\{1\}).
\end{array}\]
By Theorem~\ref{tm:nong},
an upper bound for the algebraic degree is given by (\ref{eq:GNEPdegnong}), which equals
\be\label{eq:xjointsphGNEPdegree}
\check{\mc{D}}_{ E^{(1)}}+\check{\mc{D}}_{ E^{(2)}}+\check{\mc{D}}_{ E^{(3)}}+\check{\mc{D}}_{ E^{(4)}}+\check{\mc{D}}_{ E^{(5)}}+\check{\mc{D}}_{ E^{(6)}}.
\ee

\begin{example}\label{ep:iprod}
Consider the GNEP (\ref{eq:ipjsGNEP}) where the objective functions are
\[\begin{array}{ll}
f_1(x,y)=&3x_1^3y_1y_2 - x_1^3y_2^2 - x_1^3y_2 - x_1^3 + 3x_1^2x_2y_1^2 + 2x_1^2x_2y_2- x_1^2y_1^2\\
&+ 2x_1^2y_1y_2 - x_1^2y_1 + 4x_1^2 + 7x_1x_2^2y_1^2 + 2x_1x_2^2y_1+ 2x_1x_2^2\\
& + 2x_1x_2y_1^2- 4x_1x_2y_1y_2+ 2x_1x_2y_1 + 2x_1x_2y_2 - 4x_1y_1 \\
&- 4x_1y_2^2+ 3x_1y_2 + 3x_2^3y_1^2 + 4x_2^3y_1y_2+ 2x_2^3y_1 - x_2^3y_2 + 5x_2^2y_1^2 \\
&+ 3x_2^2y_1y_2- 2x_2^2y_2^2 + 2x_2^2y_2 + 2x_2^2 - 2x_2y_1^2- x_2y_1y_2- x_2y_2^2.
\end{array}
\]
\[\begin{array}{ll}
f_2(x,y)=&3x_1^2y_1^3 - x_1^2y_1^2 - 3x_1^2y_1y_2^2 - x_1^2y_1y_2 - 2x_1^2y_2^2 - x_1^2y_2+ 3x_1x_2y_1^3\\
&- x_1x_2y_1^2y_2- x_1x_2y_1^2 - x_1x_2y_1y_2 - x_1x_2y_1 - 3x_1x_2y_2^3- 2x_1x_2y_2^2\\
&+ 4x_1x_2y_2- x_1y_1^3+ 2x_1y_1^2y_2 - x_1y_1^2 + 3x_1y_1y_2^2 - x_1y_2^3+ 2x_1y_2^2\\
& - 2x_1y_2 + 2x_2^2y_1^3 + 2x_2^2y_1^2y_2+ 2x_2^2y_1y_2 + 2x_2^2y_1 - x_2^2y_2^3+ 4x_2^2y_2^2\\
&- 2x_2^2y_2  - 2x_2y_1y_2^2 - x_2y_1y_2 + 2x_2y_2 + 3y_1^3- 3y_1^2 - y_1 + 3y_2^3- y_2^2.
\end{array}
\]
For objective functions, the multi-degrees are $d_{1,0}=(3,2)$, $d_{2,0}=(2,3)$, and
\[\begin{array}{ccc}
\check{\mc{D}}_{ E^{(1)}}=60, & \check{\mc{D}}_{ E^{(2)}}=74, & \check{\mc{D}}_{ E^{(3)}}=12, \\
\check{\mc{D}}_{ E^{(4)}}=16, & \check{\mc{D}}_{ E^{(5)}}=12, & \check{\mc{D}}_{ E^{(6)}}=16.\end{array}  \]
Therefore, the upper bound provided by Theorem~\ref{tm:nong} and (\ref{eq:xjointsphGNEPdegree}) is $190$.
For the symbolic computation, we remark that there are $6$ cases of active sets,
and all $I_{E^{(j)}}\, (j=1\ddd 6)$ has $4$ variables.
The generators and their multi-degrees are presented in the following:
\begin{itemize}
  \item $I_{E^{(1)}} = \lip g_{1,1}, \mathfrak{m}_{1,1}, \frac{\pt f_2}{\pt x_{2,1}}, \frac{\pt f_2}{\pt x_{2,2}} \rip$: $\deg(g_{1,1}) = (1,1)$, $\deg(\mathfrak{m}_{1,1}) = (2,3)$, and $\deg(\frac{\pt f_2}{\pt x_{2,1}}) = \deg(\frac{\pt f_2}{\pt x_{2,2}}) = (2,2)$.

  \item $I_{E^{(2)}} = \lip g_{1,1}, g_{2,1}, \mathfrak{m}_{1,1}, \mathfrak{m}_{2,1} \rip$: $\deg(g_{1,1}) = (1,1)$, $\deg(g_{2,1}) = (2,2)$, $\deg(\mathfrak{m}_{1,1}) = (2,3)$, and $\deg(\mathfrak{m}_{2,1}) = (2,3)$.

  \item $I_{E^{(3)}} =  \lip g_{1,1}, g_{1,2}, \frac{\pt f_2}{\pt x_{2,1}}, \frac{\pt f_2}{\pt x_{2,2}} \rip $:
  $\deg(g_{1,1}) = (1,1)$, $\deg(g_{1,2}) = (1,0)$, and $\deg(\frac{\pt f_2}{\pt x_{2,1}}) = \deg(\frac{\pt f_2}{\pt x_{2,2}}) = (2,2)$.

  \item $I_{E^{(4)}} =  \lip g_{1,1}, g_{1,2},  g_{2,1}, \mathfrak{m}_{2,1} \rip $:
  $\deg(g_{1,1}) = (1,1)$, $\deg(g_{1,2}) = (1,0)$, $\deg(g_{2,1}) = (2,2)$, and $\deg(\mathfrak{m}_{2,1}) = (2,3)$.

  \item $I_{E^{(5)}} =  \lip g_{1,1}, g_{1,3}, \frac{\pt f_2}{\pt x_{2,1}}, \frac{\pt f_2}{\pt x_{2,2}} \rip $:
  $\deg(g_{1,1}) = (1,1)$, $\deg(g_{1,3}) = (1,0)$, and $\deg(\frac{\pt f_2}{\pt x_{2,1}}) = \deg(\frac{\pt f_2}{\pt x_{2,2}}) = (2,2)$.

  \item $I_{E^{(6)}} =  \lip g_{1,1}, g_{1,3}, g_{2,1}, \mathfrak{m}_{2,1} \rip $:
  $\deg(g_{1,1}) = (1,1)$, $\deg(g_{1,3}) = (1,0)$, $\deg(g_{2,1}) = (2,2)$, and $\deg(\mathfrak{m}_{2,1}) = (2,3)$.
\end{itemize}
Using symbolic computation software {\tt Macaulay2},
we find that $V(I_{E^{(1)}})$ has $60$ points,
$V(I_{E^{(2)}})$ has $74$ points,
$V(I_{E^{(3)}})$ and $V(I_{E^{(5)}})$ has $12$ points,
and $V(I_{E^{(4)}})$ and $V(I_{E^{(6)}})$ has $16$ points (all counting multiplicities).
Besides that, all these affine varieties are pairwisely disjoint.
Thus, the Fritz-John variety has exactly $190$ points (counting multiplicities).
Interestingly, if we apply formula (\ref{eq:GNEPdeg}), i.e., we use $\hat{d}_{2,1} = (2,1)$ instead of $\check{d}_{2,1}=(0,1)$, to compute the upper bound of algebraic degree for each active set,
then we get
\[\begin{array}{ccc}
{\mc{D}}_{ E^{(1)}}=60, & {\mc{D}}_{ E^{(2)}}=106, & {\mc{D}}_{ E^{(3)}}=12, \\
{\mc{D}}_{ E^{(4)}}=20, & {\mc{D}}_{ E^{(5)}}=12, & {\mc{D}}_{ E^{(6)}}=20,\end{array}  \]
hence the upper bound given by Theorem~\ref{tm:GNEPdegreeinactive} is $230$,
which is not sharp and looser than that given by Theorem~\ref{tm:nong}.
\end{example}

\subsection{GNEPs with generic quadratic constraints}
Consider the two-player GNEP
\be\label{eq:jljqGNEP}
\begin{array}{cllcl}
    \min\limits_{x\in \re^2 }& f_1(x,y) &\vline&
    \min\limits_{y\in \re^2 }& f_2(x,y) \\
    \st & g_1(x,y)=0,&\vline& \st &g_2(x,y)=0,\\
\end{array}
\ee
where $g_1$ and $g_2$ has multi-degrees $d_{1,1}=(2,2),\ d_{2,1}=(2,2)$ respectively.
Moreover, the $\nu=(2,2)$, $m=2$, and
\[\hat{d}_{1,1}=(1,2),\ \hat{d}_{2,1}=(2,1).\]
Suppose the multi-degrees for $f_1$ and $f_2$ are $d_{1,0}$ and $d_{2,0}$.
Since all constraints are active, by Theorem~\ref{tm:GNEPdegree},
the upper bound for the algebraic degree is given by (\ref{eq:GNEPdeg}).

\begin{example}\label{ep:3player}
Consider the GNEP (\ref{eq:jljqGNEP}) defined by
\[\begin{array}{rl}
f_1(x,y)=&- 2x_1^3y_1^2 - 2x_1^3y_1 - x_1^3y_2^2 - x_1^3 - 2x_1^2x_2y_1y_2+ 3x_1^2x_2y_1\\
&- 4x_1^2x_2y_2 + 5x_1^2x_2 - x_1^2y_1^2- 2x_1^2y_1y_2+ 3x_1^2y_1- 4x_1^2y_2\\
&+ 4x_1x_2^2y_1^2+ 2x_1x_2^2y_1 + 2x_1x_2y_1- 2x_1x_2y_2+ 2x_1x_2\\
&+ 2x_1y_1^2 + 2x_1y_1y_2 - x_1y_1- x_2^3y_1y_2- 3x_2^3y_2^2+ 2x_2^3y_2+ 2x_2^3 \\
& - 4x_2^2y_1y_2+ 3x_2^2y_2^2- x_2^2y_2+ 2x_2y_1^2+ 2x_2y_1y_2+ 2x_2y_2,\\
\end{array}
\]
\[\begin{array}{ll}
f_2(x,y)=&3x_1^2y_1^2 - 2x_1^2y_1 - 3x_1^2y_2^2 + 4x_1^2y_2 - x_1x_2y_1^2- 6x_1x_2y_1y_2^2\\
&- x_1x_2y_1y_2 - x_1x_2y_2^2- 2x_1x_2y_2+ 2x_1y_1^3 + 2x_1y_1^2 - x_1y_1y_2^2\\
&- 2x_1y_2 - 3x_2^2y_1^3+ 4x_2^2y_1y_2^2 - 2x_2^2y_1+ 2x_2^2y_2 - 2x_2y_1^3- x_2y_1^2\\
&- x_2y_1+ 4x_2y_2^3 - x_2y_2^2 + 4x_2y_2 + 2y_1^3 + 4y_1^2y_2,
\end{array}
\]
\[\begin{array}{ll}
g_1(x,y)=&- 2x_1^2y_1^2 + 2x_1^2y_1y_2 + 3x_1^2y_1 + 4x_1^2y_2^2 + 2x_1^2y_2 + 3x_1^2 + x_1x_2y_1y_2\\
 &+ 3x_1x_2y_1 - x_1x_2y_2^2 + 2x_1x_2y_2 + x_1x_2 + 6x_1y_1^2 - 2x_1y_1y_2 + x_1y_1  \\
 &+ 7x_1y_2^2 + 8x_1y_2 - x_2^2y_1^2 - 5x_2^2y_1y_2 - 2x_2^2y_1 + 3x_2^2y_2^2 - 2x_2^2y_2 + 2x_2^2  \\
 &+ 3x_2y_2^2 + 2x_2y_2 + 2x_2 + 2y_1^2 + y_1y_2 + 4y_1 - 2y_2^2 - 4y_2 + 2
\end{array}
\]
\[\begin{array}{ll}
g_2(x,y)=&2x_1^2y_1y_2 + 2x_1^2y_1 + 2x_1^2y_2^2 + x_1^2y_2 - 2x_1^2+ x_1x_2y_1^2- x_1x_2y_1y_2\\
&+ 3x_1x_2y_1 - x_1x_2y_2^2+ 2x_1x_2y_2 + x_1x_2 - 2x_1y_1^2 + x_1y_1y_2\\
& - 3x_1y_1+ 2x_1y_2^2 + x_1y_2+ x_1 - 3x_2^2y_1^2+ 3x_2^2y_1y_2 - 3x_2^2y_1\\
&+ 2x_2^2y_2^2+ 3x_2^2- x_2y_1^2 + x_2y_1y_2 + 3x_2y_1- x_2y_2^2+ 3x_2y_2\\
&+ x_2 - 3y_1^2 - y_1y_2 + y_1 + 3y_2 + 5.
\end{array}
\]
For objective functions, the multi-degrees are $d_{1,0}=(3,2)$, $d_{2,0}=(2,3)$.
Therefore, the upper bound of the algebraic degree provided by (\ref{eq:GNEPdeg}) equals 296.
Since there are no inequality constraints,
there is only one case of the active set $E = (\{1\},\{1\})$.
The ideal $I_E$ has $4$ variables, and is generated by $4$ polynomials, that
\[ \begin{gathered}
I_E = \lip g_1, g_2, \mathfrak{m}_{1,1}, \mathfrak{m}_{2,1} \rip, \\
\deg(g_1) = \deg(g_2) = (2,2),\quad \deg(\mathfrak{m}_{1,1}) = (3,4),\quad 
\deg(\mathfrak{m}_{2,1}) = (4,3).
\end{gathered}\]
By the symbolic computation software {\tt Macaulay2}, 
we find that the Fritz-John variety has exactly $296$ points (counting multiplicities).
\end{example}

\subsection{The 3-player GNEP with inner product constraints}
Consider the three-player GNEP
\be\label{eq:3ipGNEP}
\begin{array}{cllcllcl}
    \min\limits_{x \in \re^2 }& f_1(x,y,z) &\vline&
    \min\limits_{y \in \re^2 }& f_2(x,y,z) &\vline&
    \min\limits_{z \in \re^2 }& f_3(x,y,z) \\
    \st & x^Ty=0, &\vline& \st & y^Tz=0, &\vline& \st & z^Tx=0,\\
\end{array}
\ee
where $g_1$, $g_2$ and $g_3$ has multi-degrees $d_{1,1}=(1,1,0),\ d_{2,1}=(0,1,1)$ and $d_{3,1}=(1,0,1)$ respectively.
Moreover, the $\nu=(2,2,2)$, $m=3$, and
\[\check{d}_{1,1} = \hat{d}_{1,1}=(0,1,0),\ \check{d}_{2,1}= \hat{d}_{2,1}=(0,0,1),\ \check{d}_{3,1}= \hat{d}_{3,1}=(1,0,0).\]
Suppose the multi-degree for $f_1$, $f_2$ and $f_3$ are $d_{1,0}$, $d_{2,0}$ and $d_{3,0}$.
Let $\check{d}_{1,0}$, $\check{d}_{2,0}$ and $\check{d}_{3,0}$ be defined as in Theorem~\ref{tm:nong}.
Since all constraints are active, by Theorem~\ref{tm:nong},
the upper bound for the algebraic degree is given by (\ref{eq:GNEPdegnong}).

\begin{example}
Consider the GNEP (\ref{eq:jljqGNEP}) where the objective functions are
\[\begin{array}{ll}
f_1(x,y,z)=&6x_1 + 3x_2 + 2x_1y_1 - 2x_2y_1 - 2x_1z_1 + 7x_1z_2 + 3x_2z_1 + 2x_2z_2 \\
&+ 2x_1^2y_1 - x_1^2y_2 + 3x_2^2y_1 - x_1^2z_2 + 2x_2^2z_2 - 2x_1^2y_1z_1 - 2x_1^2y_1z_2\\
&- 5x_1^2y_2z_1 + 3x_2^2y_1z_1 + 3x_1^2y_2z_2 - x_2^2y_1z_2 - 2x_2^2y_2z_1 - 2x_2^2y_2z_2\\
&- 3x_1x_2y_1 - x_1x_2z_1 + 3x_1x_2z_2 + 2x_1y_1z_2 + 2x_2y_1z_1 + 3x_1y_2z_2 \\
&+ 4x_2y_1z_2 + 3x_2y_2z_1 + 2x_2y_2z_2 - x_1x_2y_2z_2,
\end{array}
\]
\[\begin{array}{ll}
f_2(x,y,z)=&- 2y_2 - 2x_1y_1 + 6x_2y_1 - 2y_1y_2 + 3y_1z_2 + 2y_2z_2 \\
&- 3x_2y_1^2- x_2y_2^2+ 4y_1^2z_1 - y_2^2z_1 + 2y_2^2z_2 + 3y_1^2 + 5y_2^2  \\
&+ 4x_1y_1^2z_2+ 2x_2y_1^2z_1- 2x_1y_2^2z_2- x_2y_1^2z_2 + 2x_2y_2^2z_2    \\
&- x_1y_1y_2- 2x_1y_1z_1- x_2y_1z_1+ 3x_1y_2z_2- 3x_2y_2z_1  \\
&- y_1y_2z_1- 2y_1y_2z_2 + 2x_1y_1y_2z_1 + 2x_1y_1y_2z_2 - x_2y_1y_2z_1,
\end{array}
\]
\[\begin{array}{ll}
f_3(x,y,z)=&- 4z_2 + 2x_1z_2 - x_2z_1 - 2x_2z_2 + 3y_1z_2 + 2y_2z_1 + 3y_2z_2\\
&+ 3z_1z_2 - 4x_1z_1^2 - x_1z_2^2 - x_2z_1^2 + 2x_2z_2^2 - 3y_1z_1^2 - 2y_1z_2^2\\
&- 2y_2z_1^2 - z_1^2 - 2z_2^2 + 3x_1y_1z_2^2 - x_2y_1z_1^2 - x_2y_2z_1^2 - x_2y_2z_2^2\\
&+ 2x_1y_1z_1 + 3x_2y_2z_1 - 4x_2z_1z_2 - 2x_1y_1z_1z_2 - 3x_1y_2z_1z_2.
\end{array}
\]
For objective functions, the multi-degrees are $d_{1,0}=(2,1,1)$, $d_{2,0}=(1,2,1)$,
and $d_{3,0}=(1,1,2)$.
The algebraic degree provided by (\ref{eq:GNEPdegnong}) of this GNEP is $74$.
Since there does not exist inequality constraints,
there is only one case of the active set $E = (\{1\},\{1\},\{1\})$.
The ideal $I_E$ has $6$ variables, and is generated by $6$ polynomials, that
\[ \begin{gathered}
I_E = \lip g_1, g_2, g_3, \mathfrak{m}_{1,1}, \mathfrak{m}_{2,1}, \mathfrak{m}_{3,1} \rip, \\
\deg(g_1) = (1,1,0),\quad \deg(g_2) = (0,1,1),\quad \deg(g_3) = (1,0,1),\\
\deg(\mathfrak{m}_{1,1}) = (1,2,1),\quad \deg(\mathfrak{m}_{2,1}) = (1,1,2), \quad \deg(\mathfrak{m}_{3,1}) = (2,1,1).
\end{gathered}\]
By the symbolic computation software {\tt Macaulay2}, 
we find that the Fritz-John variety has exactly $74$ points (counting multiplicities).

Last, we remark that in this GNEP, since $\hat{d}_{i,j} = \check{d}_{i,j}$
holds for all $i=1\ddd 3$ and $j=0,1$, thus the upper bound for the algebraic degree given by (\ref{eq:GNEPdegnong})
is equal to that given by (\ref{eq:GNEPdeg}).
\end{example}

\subsection{The 2-player GNEP with a joint sphere constraint}
We present a GNEP which has infinitely many complex Fritz-John points.
Consider the two-player GNEP
\be\label{eq:jointsphere}
\begin{array}{cllcl}
    \min\limits_{x \in \re^3 }& f_1(x,y) &\vline&
    \min\limits_{y \in \re^3 }& f_2(x,y) \\
    \st & \Vert x\Vert^2+\Vert y\Vert^2 = 1 ,&\vline& \st & \Vert x\Vert^2+\Vert y\Vert^2 = 1.\\
\end{array}
\ee
In the GNEP, $\nu = (3,3)$, $m_1=m_2=1$ and $m=2$.
However, the constraining polynomials $g_{1,1}$ and $g_{1,2}$ are identical.
Recall that the affine varieties $\mc{U}_1$ and $\mc{U}_2$
are given by $g_{1,1}(x)=0$ and $g_{2,1}(x)=0$, respectively.
We have $\mc{U}_1 = \mc{U}_2$, and
\[ \cod\,\mc{U} = \cod\, \mc{U}_1\cap\mc{U}_2  = 1. \]
Also, the determinantal varieties $\mc{V}_1$, $\mc{V}_2$
has codimensions not greater than $2$, by \cite[Theorem~17.25]{Harris1992}.
So, the codimension for the Fritz-John variety cannot exceed
\[\cod\,\mc{U} + \cod\,\mc{V}_1 + \cod\,\mc{V}_2 = 5.\]
In other words, if the Fritz-John variety $\mc{W}$ is nonempty,
then $\dim \mc{W} \ge 1$ and there are infinitely many complex Fritz-John points.

\begin{example}
Consider the GNEP (\ref{eq:jointsphere}) where the objective functions
$f_1$ and $f_2$ are given as in Example~\ref{ep:ballNEP}.
Then, $d_{1,0} = (2,1)$, $d_{1,1} = (2,2)$, $d_{2,0} = (1,2)$, and $d_{2,1} = (2,2)$.
Also, since the constraints are not generic,
we have $\check{d}_{1,1} = (1,0)$ and $\check{d}_{2,1} = (0,1)$.
Since there are no inequality constraints,
there is only one case of active set $E$.
The ideal $I_E$ has $6$ variables, and is generated by $8$ polynomials, that
\[ \begin{gathered}
I_E = \lip g_1, g_2, \mathfrak{m}_{1,1}, \mathfrak{m}_{1,2}, \mathfrak{m}_{1,3}, \mathfrak{m}_{2,1}, \mathfrak{m}_{2,2}, \mathfrak{m}_{2,3} \rip, \\
\deg(g_1) = (2,2),\quad \deg(g_2) = (2,2),\\
\deg(\mathfrak{m}_{1,1}) = \deg(\mathfrak{m}_{1,2}) = \deg(\mathfrak{m}_{1,3}) = (2,1),\\
\deg(\mathfrak{m}_{2,1}) = \deg(\mathfrak{m}_{2,2}) = \deg(\mathfrak{m}_{2,3}) = (1,2).
\end{gathered}\]
By the symbolic computation software {\tt Macaulay2}, 
we find that $\dim(I_E) = 1$, which certifies that 
this GNEP has infinitely many Fritz-John points.
\end{example}

\section{Conclusions and Discussions}
This paper studies the generalized Nash equilibrium problems defined by polynomials.
We show that under some generic conditions,
the Fritz-John system of the GNEP only has finitely many complex solutions,
and all of them satisfy the KKT conditions for each player's optimization.
Moreover, when the GNEP is defined by generic polynomials,
we give a formula for the algebraic degree of the GNEP,
which counts the number of complex Fritz-John points.
When the GNEP is given by non-generic polynomials,
the algebraic degree is also studied.

For future work, an interesting question is to find an upper bound
for the number of real KKT points
when the GENP is given by polynomials with real coefficients.
Certainly, all real KKT points are in the Fritz-John variety,
so (\ref{eq:GNEPdeg}) and (\ref{eq:GNEPdeginactive}) still give upper bounds for the number of them.
However, since (\ref{eq:GNEPdeg}) and (\ref{eq:GNEPdeginactive}) count the number of complex points,
and inequalities in the KKT system are ignored,
the upper bound given by (\ref{eq:GNEPdeg}) and (\ref{eq:GNEPdeginactive}) are usually much greater than the number of real KKT points.
In general, how to find a formula to count the number of real Fritz-John or KKT points for a GNEP
is still mostly open, to the best of the authors' knowledge.
For instance, for the classical unconstrained polynomial optimization,
say, with degree $d$ and $n$ variables, we do not know
the formula to count its number of real KKT points.
In Sections~\ref{sc:pre} and \ref{sc:finite}, when the ground field is real and all polynomials have only real coefficients,
the same conclusions there still hold for each irreducible component of the complex varieties
defined by these real polynomials.
Moreover, each component of the complex variety has the same dimension as its real parts unless the real part is empty or is contained in the singular locus of the component.
In Section~\ref{sc:degree}, the number of complex solutions is the same
when the polynomials are generic with real coefficients.
The number of real solutions, however, may vary a lot.
For instance, for the very special case of minimizing real univariate polynomials of degree $3$
without any constraints,
the KKT set is given by a quadratic equation with real coefficients.
The number of real KKT points reaches the maximum in an open Euclidean set in the space of real coefficients (when the discriminant is positive),
while that number is zero for another open Euclidean set (when the discriminant is negative).
Therefore, when considering the number of real KKT points, the notion of generic polynomials makes less sense.
The maximal number of real KKT points may be attained in an Euclidean open set in the space of coefficients, but whether that maximum number equals the number of complex solutions is not clear to us.
Even if we bound degrees of all polynomials, unless they are all ones,
we do not know a general statement about the maximum number of real Fritz-John or KKT points.

Another interesting topic for future work is to count the algebraic degree of GNEPs
when their defining polynomials have sparsity.
In \cite{sottile2001enumerative}, the algebraic degrees
for sparse polynomial systems with a fixed {\it monomial support,}
i.e., the monomial powers where
the associate coefficients are nonzero, are considered.
Algebraic degrees for polynomial optimization with
a fixed monomial support are studied in \cite{lindberg2023algebraic}.
The mixed volumes of corresponding Newton polytopes are exploited in \cite{lindberg2023algebraic,sottile2001enumerative}
to give the algebraic degree for these sparse polynomial systems.
Thus, one wonders if this approach can be applied
to find algebraic degrees for GNEPs with sparsity.
More specifically, for the fixed supports of $f_i$ and $g_{i,j}$ $(i\in[N],j\in[m_i])$,
when the coefficients are generic, is it true that
the GNEP has finitely many Fritz-John points?
Does the LICQC still hold at all Fritz-John points?
Moreover, how can get a clear formula for the number of complex Fritz-John points
for generic GNEPs with a fixed support?
These are interesting questions for future work.

\section*{Acknowledgements}
Jiawang Nie is partially supported by the NSF grant DMS-2110780.
Xindong Tang was partially supported by the Start-up Fund P0038976/BD7L
during his former position at the Hong Kong Polytechnic University.

\bigskip 
\noindent
{\small \textbf{Conflict of Interest}\quad  
The authors declare that they have no conflict of interest.}

\end{document}